\documentclass[final]{siamart190516}
\usepackage[utf8]{inputenc}
\usepackage{amsmath}
\usepackage{amsfonts}
\usepackage{amssymb}
\usepackage{graphicx}

\usepackage{physics}
\usepackage{bm}
\usepackage{algpseudocode}

\graphicspath{{Figures/}}

\newcommand{\R}{\mathbb{R}}

\newcommand{\vt}[1]{\tilde{\vb{#1}}}

\DeclareMathOperator*{\argmin}{arg\,min}
\DeclareMathOperator{\spn}{span}

\newcommand{\T}{\mathsf{T}}

\DeclareMathOperator{\diag}{diag}

\title{Locally Feasibly Projected Sequential Quadratic Programming for Nonlinear Programming on Arbitrary Smooth Constraint Manifolds\thanks{Submitted November 4, 2021
\funding{K.S.S. was supported by the U.S. Department of Energy, Office of Science, Office of Advanced Scientific Computing Research, Department of Energy Computational Science Graduate Fellowship under Award Number DE-FG02-97ER25308. This report was prepared as an account of work sponsored by an agency of the United States Government. Neither the United States Government nor any agency thereof, nor any of their employees, makes any warranty, express or implied, or assumes any legal liability or responsibility for the accuracy, completeness, or usefulness of any information, apparatus, product, or process disclosed, or represents that its use would not infringe privately owned rights. Reference herein to any specific commercial product, process, or service by trade name, trademark, manufacturer, or otherwise does not necessarily constitute or imply its endorsement, recommendation, or favoring by the United States Government or any agency thereof. The views and opinions of authors expressed herein do not necessarily state or reflect those of the United States Government or any agency thereof.}}}

\author{Kevin S. Silmore\thanks{Department of Chemical Engineering, Massachusetts Institute of Technology, Cambridge, MA, 02139 (\email{silmore@mit.edu}, \email{jswan@mit.edu})} \and James W. Swan\footnotemark[2]}

\begin{document}
\maketitle

\begin{abstract}
High-dimensional nonlinear optimization problems subject to nonlinear constraints can appear in several contexts including constrained physical and dynamical systems, statistical estimation, and other numerical models. Feasible optimization routines can sometimes be valuable if the objective function is only defined on the feasible set or if numerical difficulties associated with merit functions or infeasible termination arise during the use of infeasible optimization routines. Drawing on the Riemannian optimization and sequential quadratic programming literature, a practical algorithm is constructed to conduct feasible optimization on arbitrary implicitly defined constraint manifolds. Specifically, with $n$ (potentially bound-constrained) variables and $m < n$ nonlinear constraints, each outer optimization loop iteration involves a single $\order{nm^2}$-flop factorization, and computationally efficient retractions are constructed that involve $\order{nm}$-flop inner loop iterations. A package, LFPSQP.jl, is created using the Julia language that takes advantage of automatic differentiation and projected conjugate gradient methods for use in inexact/truncated Newton steps.
\end{abstract}

\section{Introduction} We consider nonlinear constrained optimization problems of the form
\begin{equation}
\begin{array}{rl}
\displaystyle \min_{\vb x \in \R^n} & f(\vb x) \\
\mathrm{s.t.} & \vb c(\vb x) = \bm 0 \\
& \vb d^l \leq \vb d(\vb x) \leq \vb d^u\\
& \vb x^l \leq \vb x \leq \vb x^u
\end{array}
\label{eq:mixed_opt}
\end{equation}
where $f: \R^n \to \R$, $\vb c: \R^n \to \R^m$, and $\vb d: \R^n \to \R^p$ are smooth functions (i.e., of class $C^2$). Furthermore, $\vb x^l$ and $\vb x^u$ are constant box constraints on the variables, and $\vb d^l$ and $\vb d^u$ are constant bounds for the inequality constraints. The box constraints on $\vb x$ are denoted separately from those associated with $\vb d$ since they can be handled within this work in an especially efficient manner, as will be discussed. We are primarily interested in the case where $n$ is very large and $m + p$ is small compared to $n$. Such high-dimensional and/or nonlinear problems are often solved using \textit{infeasible} primal-dual interior-point methods or sequential quadratic programming methods \cite{nocedal2006, bertsekas2016}, but there are certain advantages possessed by methods that only produce feasible iterates.

First, maintaining feasibility at all steps obviates the need to use a specialized merit function; the objective function itself can be used as a merit function. Second, if early termination is desired, the final iterate returned is feasible. This is especially useful for applications in model predictive control \cite{cao2017, dontchev2019, sun2016, tenny2004}, systems engineering, and others, where time and/or computational constraints may prevent one from solving the optimization problem completely. Third, certain objective functions may only be defined when the constraints are satisfied (at least to a certain tolerance). Taking feasible steps, of course, ensures that objective function evaluations are always well defined. For example, many physical problems feature expressions with variables that are only well defined for nonnegative values and such a feasible optimization routine would eliminate practical issues experienced when evaluating such expressions at unphysical values. Finally, regarding models of physical systems, one potential benefit of feasible optimization routines not mentioned in the literature is that local minima found when starting from certain initial conditions may better correspond to those found in reality. That is, the path taken from initial iterates to resulting local optima may be more ``physical'' in nature. Admittedly, this claim lacks certain rigor, but particular problems featuring many local optima (e.g., folding of small chain molecules, energy minimization, structural optimization, etc.) may benefit from feasible optimization routines. Such practical utility may be explored more thoroughly in future work.

When the feasible set satisfying the imposed constraints represents a submanifold of Euclidean space (discussed more below), one may turn to Riemannian optimization methods and the associated rich body of literature on these methods as a possible feasible optimization framework. Early works from Gabay \cite{gabay1982}, Luenberger \cite{luenberger1972}, and others established the utility of Riemannian optimization and specifically Riemannian Newton methods. Convergence results for Riemannian Newton methods as well as a myriad of other relevant ideas from the field of manifold optimization can be found in the monograph of Absil, Mahony, and Sepulchre \cite{absil2009a}. In particular, like its traditional Euclidean counterpart, the Riemannian Newton method exhibits quadratic convergence under standard assumptions. In fact, certain manifold-based problems (e.g., the Rayleigh quotient problem) enjoy better convergence (i.e., cubic) as shown by Smith \cite{smith1994}. 

Parallel to the development of manifold-based optimization methods, numerous researchers over the past several decades have established algorithms that also maintain feasibility at all iterates within the realm of sequential quadratic programming (SQP). In particular, Lawrence and Tits \cite{lawrence2001} established a feasible sequential quadratic programming method (FSQP) that guarantees feasibility for linear constraints but not necessarily \textit{nonlinear} constraints. Later, Wright and Tenny \cite{wright2004} developed a general framework for feasibility perturbed sequential quadratic programming (FP-SQP) that relies on a so-called ``asymptotic exactness'' property for feasibility maintenance akin to the properties of \textit{retractions} from the manifold optimization literature. A notable work from Absil et al. \cite{absil2009} demonstrated the equivalence of the Riemannian Newton and FP-SQP methods given a specific choice of Lagrange multipliers at each step of the routine (viz., the coefficients of the projection of the objective function gradient onto the normal space of the constraints).

Besides those mentioned above, several general packages \cite{manopt, Manoptjl} and other algorithms \cite{adler2002, huper2004, smith1994} have been developed to perform manifold-based optimization for various applications. However, as with FP-SQP, these algorithms often rely on the imposition of specific manifolds (e.g., spheres, Stiefel manifolds, positive definite matrices, etc.) and take advantage of their special mathematical structure. In fact, it is often considered the case that methods to maintain feasibility for arbitrary constraints lacking \textit{a priori} known structure (and consequently the need to solve complicated ``inner'' problems to do so) are impractical \cite{wright2004}. In this work, we aim to dispel this notion, combining several state-of-the-art numerical techniques to develop a general computational framework for nonlinear optimization problems with largely arbitrary implicit constraints.

The primary developments of this work include the construction of practical and computationally efficient retraction methods for arbitrary implicit constraints (i.e., those that do not possess special structure that can be exploited) as well as the use of truncated Newton methods to propose feasible step directions that avoid direct calculation of Hessian matrices. In fact, as will be discussed, as long as the feasible set remains a smooth submanifold (and, empirically, even sometimes when it does not), the method can handle linearly dependent, or degenerate, constraints that break typical Karush-Kuhn-Tucker (KKT) constraint qualifications. For the strictly equality constrained problem, each step is dominated by a factorization requiring $\order{nm^2}$ flops followed by ``inner'' iterative methods requiring $\order{nm}$ flops at each inner step. It is this factorization and associated \textit{local} parameterization of the normal space at each step that underlies the name chosen for this work: Locally Feasibly Projected Sequential Quadratic Programming (LFPSQP). Crucially, for the mixed equality and inequality constrained problem, we are able to maintain the same asymptotic scaling for both factorization and inner iterative method steps with respect to the total number of constraints regardless of the number of nontrivial box constraints imposed. Finally, we explore the use of truncated Newton methods that take advantage of mixed-mode automatic differentiation in the Julia language, thereby avoiding explicit construction of Hessian matrices. The use of the \textit{projected} conjugate gradient method \cite{gould2001, nocedal2006} specifically guarantees that proposal directions lie in the tangent space of the constraint manifold at any imposed termination residual tolerance. A software package in Julia, named \texttt{LFPSQP.jl}, is created to implement these ideas \cite{silmorelfpsqp}. To our knowledge, all of these features are not currently available in an existing code and should prove to be useful for applications that may benefit from feasible iterates and the lack of requirement to specify explicit gradients, Jacobians, or Hessian matrices.

\section{Background}

Given a set of equality constraints, denoted ``$\vb c(\vb x) = \bm 0$'', where $\vb c : \R^n \to \R^m$ is a differentiable map (assumed to be of class $C^2$ throughout), a basic result in differential geometry states that a sufficient condition for the set $\mathcal M = \{ \vb x \mid \vb c(\vb x) = \bm 0 \}$ to be a submanifold of $\R^n$ is that the Jacobian, $\vb J$, of $\vb c$ be full-rank at all points $\vb x \in \mathcal M$. Furthermore, as a submanifold of Euclidean space, the typical Euclidean inner product can be employed as the Riemannian metric, which will be assumed throughout this work. Given the assumption that $\mathcal M$ is indeed a submanifold, techniques from Riemannian optimization can be used to construct a globally convergent optimization algorithm to minimize a smooth function $f: \mathcal M \to \R$ defined on the constraint set, or feasible set, $\mathcal M$. Following Absil et al. \cite{absil2009a}, such an optimization routine in general relies on the iterative generation of a ``suitable'' (i.e., gradient-related) step tangent to the constraint manifold followed by a ``retraction'' back onto the manifold in a way that guarantees sufficient decrease of the function, such as with a line search (employed in this work) or a trust region procedure \cite{absil2007}. This skeleton of an algorithm is shown as Algorithm \ref{alg:main}, where \textsc{LineSearch} is a given line search routine that depends on the objective function and a function that maps ``step length'' to a new proposal. In this work, because $\mathcal M$ is always assumed to be a submanifold of Euclidean space, the tangent space of $\mathcal M$ at a particular point $\vb x$, denoted $T_{\vb x} \mathcal M$, can be identified with the linear subspace $\{ \vb v \in \R^n \mid \vb J(\vb x) \vb v = \bm 0 \}$. Geometrically, when translated to the point on $\mathcal M$ at which it is associated, this space can be visualized as the tangent hyperplane to the feasible set at $\vb x$ and can be identified with $\mathcal M$ itself in the special case that $\vb c$ is an affine map. As discussed more below, search directions will be generated either via a projected gradient or an inexact Newton scheme that employs second-order information of the objective function and the constraints.

\begin{algorithm}
\caption{Optimization outer loop}
\label{alg:main}
\begin{algorithmic}[1]
\Procedure{Optimize}{$\vb x^0$, $f$, $R$}
\State $i \gets 0$
\While{Termination criterion unsatisfied}
\State Generate suitable search direction $\Delta \vb x \in T_{\vb x^i} \mathcal M$
\State Line search for a new iterate: $\vb x^{i+1} \gets \textsc{LineSearch}(f, \alpha \mapsto R_{\vb x^i}(\alpha \Delta \vb x))$
\State $i \gets i + 1$
\EndWhile
\State \textbf{return} $\vb x^{i}$ 
\EndProcedure
\end{algorithmic}
\end{algorithm}

The termination criteria used by LFPSQP include a tolerance for the change in objective function value between outer steps, a tolerance for the change in the Euclidean distance between iterates, a tolerance for the norm of the projected gradient (essentially the norm of the typical Lagrangian used in SQP routines), and a maximum allowable number of iterations. All of these tolerances can be adjusted by the user.

The retraction map \cite{absil2009a} is the fundamental mechanism in the optimization routine responsible for maintaining feasibility at all iterates.
\begin{definition}
A \emph{retraction} is a map, $R_{\vb x}: T_{\vb x} \mathcal M \to \mathcal M$ such that $R_{\vb x}(\bm 0) = \vb x$ and that for every $\vb v \in T_{\vb x} \mathcal M$, $\left. \dv{t} R_{\vb x}(t \vb v) \right|_{t=0} = \vb v$.
\end{definition}
The last condition guarantees that the retraction agrees locally with the constraint manifold up to first-order. A second-order retraction is defined analogously.
\begin{definition}
A retraction is \emph{second-order} if $\left. \dv[2]{t} R_{\vb x}(t \vb v) \right|_{t=0} \in N_{\vb x} \mathcal M$ for all $\vb v \in T_{\vb x} \mathcal M$.
\end{definition}
Here, $N_{\vb x} \mathcal M$ is the normal space of the constraint manifold $\mathcal M$ at a point $\vb x$ and can be identified with the linear subspace $\{\vb v \in \R^n \mid \vb v \in \mathrm{im}(\vb J^\T(\vb x)) \}$, where $\mathrm{im}(\vb J^\T(\vb x))$ is the image, or column space, of $\vb J^\T(\vb x)$. Together, $T_{\vb x} \mathcal M$ and $N_{\vb x} \mathcal M$ span all of $\R^n$. Second-order retractions agree with the Riemannian exponential up to second-order in the distance from the zero vector of the tangent space, $\bm 0_{\vb x}$ \cite{absil2012}. In the following sections, we will describe the construction of two computationally efficient routines for second-order retractions that take advantage of a local parameterization of the normal space at each iterate \cite{absil2012, oustry1999}. It is thus important to note that retractions in general, and especially those considered in this work, may \textit{not} necessarily be well-defined on the entirety of $T_{\vb x} \mathcal M$ but rather within a certain neighborhood of $\bm 0_{\vb x}$.

The gradient of a function defined on a Riemannian manifold can be defined intrinsically with respect to the Riemannian metric \cite[ch. 3]{absil2009a}.
\begin{definition}
The \emph{gradient} of a smooth function $f : \mathcal M \to \R$ at a point $\vb x$ is the unique element $\mathrm{grad}\, f(x) \in T_{\vb x} \mathcal M$ satisfying $\langle \mathrm{grad}\, f(x), \vb v \rangle = \mathrm{D}f(x)[\vb v]$ for all $\vb v \in T_{\vb x} \mathcal M$, where $\mathrm{D}f(x)$ is the differential of $f$ at $\vb x$.
\end{definition}
For submanifolds of Euclidean space employing the typical Euclidean metric, the gradient of $f$ can be identified as the projection of the Euclidean gradient (denoted $\grad f$ and not to be confused with the Levi-Civita connection) onto the tangent space of $\mathcal M$ at $\vb x$. Likewise, the Riemannian Hessian \cite[ch. 5]{absil2009a} for a submanifold of Euclidean space can be defined similarly with projections involving the Euclidean Hessian of $f$, $\grad \grad f$, as well as the Hessians of the constraint functions, which, together, account for variation of $f$ in the tangent space and variation of $f$ due to curvature of the constraints.

\section{Equality Constraints}

In this section, we are only concerned with presence of equality constraints in an optimization problem of the form:
\begin{equation}
\begin{array}{rl}
\displaystyle \min_{\vb x \in \R^n} & f(\vb x) \\
\mathrm{s.t.} & \vb c(\vb x) = \bm 0.
\end{array}
\label{eq:opt_equality}
\end{equation}
The addition of inequality constraints will be considered in the next section. Again, it is assumed that $f$ and $\vb c$ are smooth maps and that the set $\mathcal M = \{ \vb x \mid \vb c(\vb x) = \bm 0 \}$ is a submanifold of Euclidean space.

\subsection{Automatic differentiation}

As detailed more thoroughly in the following subsections, generating a search direction requires access to the gradient of the objective function, $\grad f$, the Jacobian of the constraints, $\vb J$, and the \textit{action} of the Hessian of the Lagrangian (optional if the inclusion of second-order information is not desired). While the user can code these functions explicitly and pass them to the main LFPSQP routine, they are calculated by default using automatic differentiation. With a computer language that allows for it, automatic differentiation of functions has proven to be a very effective tool for numerical optimization routines \cite{griewank2008}, eliminating the need to calculate derivatives by hand (and possible human errors associated with it) or the use of lower accuracy finite difference approximations. In this work, we take advantage of the rich software ecosystem in Julia, employing reverse-mode automatic differentiation with the Julia package ReverseDiff.jl \cite{revels} as well as mixed-mode automatic differentiation using a combination of RerverseDiff.jl and ForwardDiff.jl \cite{revels2016}. Specifically, $\grad f(\vb x)$ and $\vb J(\vb x)$ are calculated using reverse-mode automatic differentiation since reverse-mode is often faster in the case of functions with many inputs and few outputs (i.e., when $n$ is much larger than $m$, which is assumed to be the case here). The Lagrangian function associated with the optimization problem \ref{eq:opt_equality} is defined as
\begin{equation}
\mathcal L(\vb x, \bm \lambda) = f + \bm \lambda \vdot \vb c(\vb x).
\end{equation}
The Hessian of the Lagrangian at a given pair, $(\vb x, \bm \lambda)$, is the matrix:
\begin{equation}
\vb W(\vb x, \bm \lambda) = \grad \grad f(\vb x) + \sum_{k=1}^m \lambda_k \grad \grad c_k(\vb x),
\end{equation}
where ``$\grad$'' denotes a derivative with respect to $\vb x$ and $c_k$ is the $k$-th element of $\vb c$. The action, $\vb W(\vb x, \bm \lambda) \vb v$, for some vector $\vb v \in \R^n$, can be written as
\begin{equation}
\vb W(\vb x, \bm \lambda) \vb v = \left. \dv{a} \left( \grad f(\vb x + a \vb v) + \vb J^\T(\vb x + a \vb v) \bm \lambda \right) \right|_{a=0}.
\end{equation}
Thus, mixed-mode differentiation --- reverse-mode to calculate the gradient of $f$ and Jacobian of $\vb c$, and forward-mode to calculate the derivative of those functions with respect to $a$ --- can be used to calculate the action $\vb W(\vb x, \bm \lambda) \vb v$. Practically, this scheme is possible because ReverseDiff can calculate derivatives of functions with respect to ``dual'' input types used by ForwardDiff and not just primitive floating point inputs.

\subsection{Step Generation and Line Search}

At each iteration $i$ of the outer optimization loop, the Jacobian of the constraints, $\vb J(\vb x^i)$, is factored in order to construct a local orthogonal basis for the normal space, $N_{\vb x^i} \mathcal M$. We choose to use a \textit{thin} singular value decomposition (SVD) for such a factorization due to its rank-revealing nature \cite{golub2013} such that
\begin{equation}
\vb U \bm \Sigma \vb V^\T = \vb J^\T(\vb x^i),
\end{equation}
where $\vb U \in \R^{n \times m}$ and $\vb V \in \R^{m \times m}$ are matrices with orthonormal columns and $\bm \Sigma$ is a diagonal matrix. In general, this decomposition requires $\order{n m^2}$ flops to compute. The dependence of the factorization on the $i$th iterate is implied and will be assumed throughout. One could also consider another decomposition, such as a rank-revealing QR decomposition, but we ultimately use the SVD given that it is more numerically stable and does not take significantly more time to compute than a rank-revealing QR decomposition in practice, especially when $m \ll n$. This decomposition, though, is essentially a specific construction of the tangential parameterizations considered by Absil et al. \cite{absil2012} and Oustry \cite{oustry1999}. As $m$ grows and is commensurate to $n$ in size, such a decomposition may not be practical to compute, especially if $\vb J$ is sparse, as is often the case for many problems encountered in practice. Although not considered in this work, potential methods to take advantage of the sparsity of $\vb J$ are discussed below. 

In the event that some constraints are not linearly independent, given a user-specified tolerance, $\epsilon_\mathrm{rank}$, the numerical rank of $\vb J$, $\rank_{\epsilon_\mathrm{rank}}(\vb J)$, is defined as the number of singular values greater than $\epsilon_\mathrm{rank}$. The first $\rank_{\epsilon_\mathrm{rank}}(\vb J)$ columns of $\vb U$, then, provide a basis for the normal space of the constraint manifold (assuming the feasible set is still indeed a submanifold despite the degeneracy). Thus, given the SVD of $\vb J(\vb x^i)$, the projected gradient of $f$ onto the tangent space at $\vb x^i$ (equal to the Riemannian gradient of $f$) can be easily calculated as
\begin{equation}
\mathrm{grad}\, f(\vb x^i) = P_{\vb x^i}[\grad f(\vb x^i)] = (\vb I - \vb U_{:,1:r} \vb U_{:,1:r}^\T)\grad f(\vb x^i),
\end{equation}
where $P_{\vb x^i}$ is the projection operator onto the tangent space at $\vb x^i$, $\vb I$ is the identity matrix, and $\vb U_{:,1:r}$ represents the first $r = \rank_{\epsilon_\mathrm{rank}}(\vb J)$ columns of $\vb U$. This projection operation with $\vb U$ only requires $\order{nm}$ flops. A sketch of the algorithm to generate a gradient search direction is shown in Algorithm \ref{alg:gradient} and can be used in Algorithm \ref{alg:main} along with a suitable line search.

\begin{algorithm}
\caption{Generation of the gradient search direction}
\label{alg:gradient}
\begin{algorithmic}[1]
\Procedure{GenerateGradientDirection}{$\vb x^i$, $f$, $\vb J$}

\State Generate SVD decomposition of $\vb J^\T(\vb x^i)$
\State $\Delta \vb x \gets P_{\vb x^i}[-\grad f(\vb x^i)]$

\State \textbf{return} $\Delta \vb x$

\EndProcedure
\end{algorithmic}
\end{algorithm}

For line searches, we implement an Armijo backtracking line search as well as McCormick's ``exact'' golden ratio bisection line search \cite[sec. 5.4]{mccormick1983}. The Armijo backtracking line search finds the minimum integer, $k \geq 0$, such that
\begin{equation}
\begin{split}
f(\vb x^i) - f(R_{\vb x^i}(\alpha_0 s^k \Delta \vb x)) &\geq -\sigma \alpha_0 s^k  P_{\vb x^i}[\grad f(\vb x^i)] \vdot \Delta \vb x \\
&= -\sigma \alpha_0 s^k  \grad f(\vb x^i) \vdot \Delta \vb x,
\end{split}
\end{equation}
where $\alpha_0 > 0$, $s \in (0, 1)$, and $\sigma \in (0, 1)$ are user-specified parameters representing an initial step length guess, a reduction factor, and a suitable criterion for sufficient decrease, respectively. The second equality follows from the fact that $\Delta \vb x$ lies in the tangent space of the constraint manifold such that
\begin{equation*}
P_{\vb x^i}[\grad f(\vb x^i)] \vdot \Delta \vb x = \grad f(\vb x^i) \vdot P_{\vb x^i}[\Delta \vb x] = \grad f(\vb x^i) \vdot \Delta \vb x.
\end{equation*}
Values that are practical for many problems are $\alpha_0 = 1$, $s = 1/2$, and $\sigma = 10^{-4}$.

McCormick's bisection procedure \cite[sec. 5.4]{mccormick1983} is robust and is guaranteed to yield a local minimizer of $f$ along the ``arc'' $\alpha \mapsto R_{\vb x^i}(\alpha \Delta \vb x)$ as long as an optimum for $\alpha$ exists or if the range of $\alpha$ over which the retraction $R_{\vb x^i}(\alpha \Delta \vb x)$ is defined is bounded. Unlike typical bisection procedures over a given interval, McCormick's bisection search features an upper bounding procedure such that no \textit{a priori} specified search interval for $\alpha$ is required. The convergence criterion implemented in the package we develop bisects the interval until the distance between the two endpoints is less than $10^{-6} \norm{\Delta \vb x}_2$ and then selects the endpoint associated with the greatest decrease in $f$. While this exact line search procedure often requires many more function evaluations than Armijo backtracking, it may be desirable for certain ill-conditioned objective functions for which Armijo backtracking leads to inefficient stepping.

Together, this gradient search direction along with either of these two line search procedures yields a globally convergent algorithm if the additional standard assumption that the set $\{ \vb x \in \mathcal M \mid f(\vb x) \leq f(\vb x^0) \}$ is bounded holds \cite[corollary 4.3.2]{absil2009a}.

Given that $\mathcal M$ is a submanifold of Euclidean space, the \textit{intrinsic} Riemannian Newton search direction can be calculated by solving the following (\textit{extrinsic}) linear saddle point system \cite{absil2009}:
\begin{equation*}
\mqty[\vb W(\vb x^i, \bm \lambda^i) & \vb J^\T(\vb x^i) \\ \vb J(\vb x^i) & \bm 0] \mqty[\Delta \vb x \\ \Delta \bm \lambda] = \mqty[-P_{\vb x^i}[\grad f(\vb x^i)] \\ \bm 0],
\end{equation*}
where $\bm \lambda^i$ solves the linear system:
\begin{equation*}
\vb J(\vb x^i) \vb J^\T(\vb x^i) \bm \lambda^i = -\vb J(\vb x^i) \grad f(\vb x^i).
\end{equation*}
As Absil et al. note \cite{absil2009}, this is the same system of equations typically solved in equality-constrained SQP routines with the exception that $\bm \lambda^i$ is set to be the coefficients of the projection of $\grad f(\vb x^i)$ onto the normal space and not maintained/updated independently during the course of the algorithm. $\Delta \bm \lambda$ in this context, then, is rather useless outside of the saddle point system solve. It is also worth noting that, at a solution, $\vb W(\vb x^i, \bm \lambda^i)\Delta \vb x + \vb J^\T(\vb x^i)\Delta \bm \lambda = P_{\vb x^i}[\vb W(\vb x^i, \bm \lambda^i)]\Delta \vb x$, which can be identified with the action of the intrinsic Riemannian Hessian on the tangent space vector $\Delta \vb x$.

Given the SVD decomposition of $\vb J^\T(\vb x^i)$, $\bm \lambda^i$ can readily be calculated as
\begin{equation}
\bm \lambda^i = -\vb V_{:,1:r} \bm \Sigma^{-1}_{1:r,1:r} \vb U_{:,1:r}^\T \grad f(\vb x^i),
\label{eq:lambda_equality}
\end{equation}
where $r$, again, is the numerical rank of $\vb J(\vb x^i)$ and the ``slice'' notation for $\vb V$, $\bm \Sigma$, and $\vb U$ represents the truncated pseudoinverse employing only the first $r$ singular value-vector triples. Likewise, since $\vb U_{:,1:r}$ spans the same space as $\vb J^\T(\vb x^i)$, the saddle point system can be reformulated using the more numerically stable matrix $\vb U_{:,1:r}$ as:
\begin{equation}
\mqty[\vb W(\vb x^i, \bm \lambda^i) & \vb U_{:,1:r} \\ \vb U_{:,1:r}^\T & \bm 0] \mqty[\Delta \vb x \\ \Delta \bm \lambda] = \mqty[-P_{\vb x^i}[\grad f(\vb x^i)] \\ \bm 0].
\label{eq:equality_saddle}
\end{equation}
Here, abusing notation slightly, $\Delta \bm \lambda$ is used again in the solution even though it is paired with a different basis for the tangent space (i.e., the columns $\vb U_{:,1:r}$) because it is not used in the context of LFPSQP.

\begin{algorithm}
\caption{Generation of the inexact Newton search direction}
\label{alg:newton_step}
\begin{algorithmic}[1]
\Procedure{GenerateInexactNewtonDirection}{$\vb x^i$, $f$, $\vb J$, $\delta$}

\State Generate SVD decomposition of $\vb J^\T(\vb x^i)$
\State Solve equation \ref{eq:equality_saddle} to an absolute residual tolerance of $\delta$

\State \textbf{return} $\Delta \vb x$

\EndProcedure
\end{algorithmic}
\end{algorithm}

Algorithm \ref{alg:newton_step} is the skeleton of the subroutine to generate a search direction that takes advantage of second-order information. Inexact, or truncated, Newton routines have proven to be very useful in the context of optimization \cite{byrd2008, dembo1982, nash2000} as well as Riemannian optimization \cite{absil2007, aihara2017, ma2021, zhang2010, zhao2018a} as they avoid the need to construct Hessians explicitly and avoid a large of amount of computational effort solving linear systems completely at each step of the outer optimization loop. In this work, we employ the projected conjugate gradient method \cite{gould2001, nocedal2006} and mixed-mode automatic differentiation to solve equation \ref{eq:equality_saddle} iteratively. Without a good preconditioner known \textit{a priori} for the general problem, we use an adapted version of the non-preconditioned Algorithm 6.2 from ref. \cite{gould2001}, as shown in Algorithm \ref{alg:projcg}.

\begin{algorithm}
\caption{Projected conjugate gradient with orthonormal constraints}
\label{alg:projcg}
\begin{algorithmic}[1]
\Procedure{ProjCG}{$\vb A$, $\vb U$, $\vb b$, $\delta$}
\Comment Solves $\vb A \Delta \vb x + \vb U \Delta \bm \lambda = \vb b$, $\vb U^\T \Delta \vb x = \bm 0$.

\State $\Delta \vb x \gets \bm 0$
\State $\vb r, \vb g \gets -(\vb I - \vb U \vb U^\T)\vb b$
\State $\vb p \gets -\vb g$

\While{$\norm{\vb r}_2 > \delta$}

	\State $\vb q \gets \vb A \vb p$
	\\
	\If{$\vb p \vdot \vb q \leq 0$}
		\Comment Nonpositive curvature direction
		\State $\Delta \vb x \gets \vb p / \norm{\vb p}_2$
		\State \textbf{break}
	\EndIf
	\\
	\If{$\vb r \vdot \vb g \leq 0$}
		\State \textbf{return error}
	\EndIf
	\\
	\State $\alpha \gets \vb r \vdot \vb g / (\vb p \vdot \vb q)$
	\State $\Delta \vb x \gets \Delta \vb x + \alpha \vb p$
	\State $\vb r^+ = \vb r + \alpha \vb q$
	\State $\vb g^+ = (\vb I - \vb U \vb U^\T)\vb r^+$
	\State $\beta \gets \vb r^+ \vdot \vb g^+ / (\vb r \vdot \vb g)$
	\State $\vb p \gets -\vb g^+ + \beta \vb p$
	\State $\vb g, \vb r \gets \vb g^+$
	
\EndWhile

\State \textbf{return} $\Delta \vb x$
\EndProcedure
\end{algorithmic}
\end{algorithm}

Importantly, updating the residual, $\vb r$ with $\vb g^+$ in line 22 of Algorithm \ref{alg:projcg} prevents numerical roundoff errors from accruing. Unlike Algorithm 6.2 from ref. \cite{gould2001}, Algorithm \ref{alg:projcg} does not apply so-called iterative refinement during projection since the use of the orthonormal matrix $\vb U$ in the projection is optimally conditioned. Furthermore, if a negative curvature direction is detected (and $\vb A$ is determined to be indefinite or negative definite), that direction is returned as a search direction. This is a significant advantage of using the projected conjugate gradient algorithm in the context of optimization compared to other iterative algorithms (e.g., GMRES) or even full linear solves since the Hessian of the Lagrangian may not be positive semidefinite at many of the iterates that are far away from a local optimum. Some nonlinear solvers, such as IPOPT \cite{wachter2006}, heuristically add a small positive multiple of the identity matrix and try to restore feasibility in order to ensure positive definiteness of the Hessian of the Lagrangian, but such schemes are not necessary in LFPSQP. Overall, the iterative projected conjugate gradient algorithm should be robust for a broad variety of nonlinear problems.

As suggested by Dembo et al. \cite{dembo1982}, the absolute residual tolerance, $\delta$, used for inexact Newton solves at the $i$-th iterate is set as:
\begin{equation}
\delta = \kappa \min\left(1, \frac{\norm{\grad f(\vb x^i)}_2}{\norm{\grad f(\vb x^{i-1})}_2}\right) \norm{\grad f(\vb x^i)}_2
\end{equation}
and prevents ``oversolving'' far away from critical points. The parameter $\kappa \in (0,1)$ is user-specified and set to $1/2$ by default. If the Hessian of the Lagrangian is positive definite in the tangent space at a local optimum to which iterates converge, then with this choice of tolerance, it can be shown that iterates converge superlinearly (at least quadratically) to the local optimum \cite[thm. 8.2.1]{absil2009a}.

With methods for step generation and line searches specified, the following sections detail the construction of retractions that are called by the line search routines.

\subsection{Retractions}

Two computationally efficient retractions are constructed for arbitrary implicitly defined submanifolds. The first is based on quasi-Newton iteration, and the second is a numerical ``inner'' optimization routine for projection. Both rely on a user-specified parameter, $\epsilon_c$, that determines the tolerance to which the constraints are maintained (i.e., $\norm{\vb c(\vb x)}_\infty < \epsilon_c$). The quasi-Newton retraction requires the Jacobian of $\vb c$ to be full-rank, whereas the projection retraction does not and is globally convergent. Practically, the projection retraction is the default retraction chosen by the LFPSQP code. If one chooses to use the quasi-Newton retraction and the algorithm detects that the Jacobian is not numerically full-rank via the SVD, then the projection retraction is selected automatically regardless.

\subsubsection{Quasi-Newton Retraction}

Given the SVD decomposition of $\vb J^\T(\vb x^i)$ at the $i$-th step of the outer optimization loop, assume $\vb J^\T$ is full-rank. Then, one can consider the retraction (identifying $\Delta \vb x \in T_{\vb x^i} \mathcal M$ with a vector in $\R^n$):
\begin{equation}
R_{\vb x^i}: \Delta \vb x \mapsto \vb x^i + \Delta \vb x + \vb U \vb w, \quad \vb w \text{ s.t. } \vb c(\vb x^i + \Delta \vb x + \vb U \vb w) = \bm 0.
\label{eq:equality_newton_retraction}
\end{equation}
Such a retraction was considered by Gabay \cite{gabay1982} and is called the \textit{orthographic retraction} by Absil et al. \cite{absil2012} since it projects a point in the tangent plane orthographically (i.e., normal to the tangent space at $\vb x^i$) back onto the constraint manifold. Such a retraction is well defined for tangent vectors in an open neighborhood of $\bm 0_{\vb x^i} \in T_{\vb x^i} \mathcal M$ \cite[sec. 4.1.3]{absil2009a}). Additionally, this retraction is second-order \cite{absil2012}.

In order to computationally realize the retraction in equation \ref{eq:equality_newton_retraction} and calculate $\vb w$, we turn to Broyden's ``good'' quasi-Newton method \cite{broyden1973}. When $\Delta \vb x = \bm 0$, $\vb w = \bm 0$ yields a solution, and the Jacobian of $\vb c$ with respect to $\vb w$ is given by
\begin{equation}
\left. \frac{\partial \vb c(\vb x^i + \vb U \vb w)}{\partial \vb w} \right|_{\vb w = \bm 0} = \vb V \bm \Sigma.
\end{equation}
For small values of $\norm{\Delta \vb x}_2$, one should expect that the matrix $\vb B = \bm \Sigma^{-1} \vb V^\T$, which is readily computed, should be a good estimate for the true Jacobian inverse at the solution if $\vb J$ is continuous. Thus, Broyden's ``good'' quasi-Newton method is conducted with an initial inverse Jacobian guess of $\vb B = \bm \Sigma^{-1} \vb V^\T$ and updated by a rank-1 matrix at each iteration until convergence (i.e., $\norm{\vb c(\vb x^i + \Delta \vb x + \vb U \vb w)}_\infty < \epsilon_c$) is achieved. Each Broyden rank-1 update of the approximate inverse Jacobian, $\vb B$, requires $\order{m^2}$ flops, and multiplication of $\vb U$ by $\vb w$ at each iteration requires $\order{nm}$ flops. Assuming a small number of quasi-Newton iterations are required relative to $m$, the overall running time of the retraction is $\order{nm}$. Another advantage of using such a quasi-Newton method is the avoidance of having to recalculate the $m \times n$ Jacobian, $\vb J$, at every intermediate step, which may be a costly calculation compared to the basic matrix multiplications and rank-1 updates exhibited by Broyden's method. The full scheme is shown in Algorithm \ref{alg:equality_newton_projection}.

\begin{algorithm}
\caption{Quasi-Newton orthographic retraction}
\label{alg:equality_newton_projection}
\begin{algorithmic}[1]
\Procedure{QNRetract}{$\vb x^i$, $\Delta \vb x$, $\vb c$, $\vb U$, $\bm \Sigma$, $\vb V$, $\epsilon_c$, $k_\mathrm{max}$}

\State $\vb B \gets \bm \Sigma^{-1} \vb V$
\State $\vt x \gets \vb x^i + \Delta \vb x$
\State $\vb x^{i+1} \gets \vt x$

\State $k \gets 0$
\State $\vb c^k \gets \vb c(\vb x^{i+1})$

\While{$\norm{\vb c^k}_\infty > \epsilon_c$ and $k < k_\mathrm{max}$}
	\State $\Delta \vb w \gets - \vb B \vb c(\vb x^{i+1})$
	\State $\vb x^{i+1} \gets \vb x^{i+1} + \vb U \Delta \vb w$
	\State $\vb c^{k+1} \gets \vb c(\vb x^{i+1})$
	\State $\Delta \vb c \gets \vb c^{k+1} - \vb c^k$
	\State $\vb u \gets \Delta \vb w - \vb B \Delta \vb c$
	\State $\vb v \gets \vb B^\T \Delta \vb w$
	\State $\vb B \gets \vb B + \frac{1}{\vb v \vdot \Delta \vb c} \vb u \vb v^\T$
	\State $k \gets k + 1$
\EndWhile
\\
\If{$k = k_\mathrm{max}$}
	\Comment Convergence not established by $k_\mathrm{max}$ iterations
	\State \textbf{return error}
\EndIf
\\
\State \textbf{return} $\vb x^{i+1}$

\EndProcedure
\end{algorithmic}
\end{algorithm}

While Broyden's method is well known to be \textit{locally} superlinearly convergent under certain assumptions \cite{broyden1973}, it is not guaranteed that the above described procedure will converge for certain proposed steps $\Delta \vb x$. This is especially true if $\norm{\Delta \vb x}_2$ is large relative to a characteristic inverse curvature of $\vb c$, in which case there may not exist \textit{any} solutions to $\vb c(\vb x^i + \Delta \vb x + \vb U \vb w) = \bm 0$. Thus, while the quasi-Newton retraction may be useful for many problems, it may not be practical for all problems, in which case one should turn to the more robust projection retraction below.

\subsubsection{Projection Retraction}

The projection retraction is defined as
\begin{equation}
R_{\vb x^i} : \Delta \vb x \mapsto \hat{\vb x}
\end{equation}
where, letting $\vt x = \vb x^i + \Delta \vb x$ (again identifying $\Delta \vb x \in T_{\vb x^i} \mathcal M$ with a vector in $\R^n$),
\begin{equation}
\begin{array}{rl}
\hat{\vb x} \in \displaystyle \argmin_{\vb x \in \R^n} & \frac{1}{2}\norm{\vb x - \vt x}_2^2 \\
\mathrm{s.t.} & \vb c(\vb x) = \bm 0. \\
\end{array}
\end{equation}
Assuming $\vb c(\vb x^i) = \bm 0$, such a projection retraction is well defined. Furthermore, if $\mathcal M = \{\vb x \in \R^n \mid \vb c(\vb x) = \bm 0 \}$ is indeed a smooth submanifold of Euclidean space, then the retraction is second-order \cite{absil2012}. In order to solve the above optimization problem, we use the classical unconstrained quadratic penalty method \cite{bertsekas2016, mccormick1983, gould1989} involving a sequence of minimization problems indexed by a penalty parameter, $\mu$:
\begin{equation}
\hat{\vb x}_\mu \in \argmin_{\vb x \in \R^n} \frac{\mu}{2}\norm{\vb x - \vt x}_2^2 + \frac{1}{2}\norm{\vb c(\vb x)}_2^2 = \argmin_{\vb x \in \R^n} \frac{1}{2}\norm{\vb x - \vt x}_2^2 + \frac{1}{2\mu}\norm{\vb c(\vb x)}_2^2.
\label{eq:penalty_projection_opt}
\end{equation}
It is known that for a sequence $(\mu^k)_{k \in \mathbb N}$ such that $\mu^k \to 0$, if $\hat{\vb x}_\mu$ and $\hat{\vb x}$ are unique solutions to their respective problems, then $\hat{\vb x}_\mu \to \hat{\vb x}$. Unlike a more sophisticated technique, such as an augmented Lagrangian method, that could exhibit better theoretical convergence properties \cite{bertsekas1976}, the quadratic penalty method does not require the Jacobian of $\vb c$ to be full-rank at an optimum to exhibit convergence to $\hat{\vb x}$. We thus choose to use the quadratic penalty method in part for this robustness against degeneracy of the constraints \cite{dowling2015, anitescu2000}. Furthermore, although it is often cited that ill-conditioning prevents the quadratic penalty from being useful without modifications, proposed steps $\vt x$ often remain close to the feasible set, and we find that such ill-conditioning in practice does not induce numerical issues. If numerical issues do happen to arise and the retraction algorithm does not converge, then a line search routine in LFPSQP will shrink the proposed step, bringing $\vt x$ even closer to $\vb x^i$ and consequently closer to the feasible set (again assuming that $\vb x^i \in \mathcal M$).

Let $\hat{\vb x}^{k}$ represent the current iterate in the inner optimization routine for the projection retraction. Consider the placement of the penalty term in equation \ref{eq:penalty_projection_opt} on the first squared distance term, and let
\begin{equation}
\phi_{\mu}(\vb x) = \frac{\mu}{2}\norm{\vb x - \vt x}_2^2 + \frac{1}{2}\norm{\vb c(\vb x)}_2^2.
\end{equation}
Instead of solving problem \ref{eq:penalty_projection_opt} exactly, a single Gauss-Newton step in the direction $\vb p \in \R^n$ satisfying
\begin{equation}
\left( \vb J^\T(\hat{\vb x}^{k}) \vb J(\hat{\vb x}^{k}) + \mu^k \vb I \right) \vb p = -\left( \vb J^\T(\hat{\vb x}^{k}) \vb c(\hat{\vb x}^{k}) + \mu^k(\hat{\vb x}^{k} - \vt x) \right)
\label{eq:penalty_lm}
\end{equation}
and of a length satisfying a sufficient decrease in the unconstrained objective function of problem \ref{eq:penalty_projection_opt} via an Armijo line search on $\phi_{\mu^k}$ is taken to update $\hat{\vb x}^{k}$ (though we find often find it to be unnecessary in practice). Note the similarity of the update in equation \ref{eq:penalty_lm} compared to the classical Levenberg-Marquardt update for solving nonlinear equations, which is known to converge quadratically to a solution under certain assumptions on $\vb c$ \cite{yamashita2001}. As such, we update $\mu$ analogously by setting it to the current norm of the deviation from the feasible set as $\mu^{k+1} = \norm{\vb c(\hat{\vb x}^{k+1})}_2$, which appears to work well in practice. Furthermore, instead of directly inverting the approximate positive-definite Hessian in equation \ref{eq:penalty_lm}, we use an iterative conjugate gradient routine to solve for $\vb p$ to a small residual tolerance of $\epsilon_c$. Again, because the original proposal $\vt x$ is close to the feasible set, only several iterations of conjugate gradient are necessary despite the ill-conditioning of the approximate Hessian (i.e., $\kappa(\vb J^\T \vb J + \mu^k \vb I) = (\sigma_J^2 + \mu^k)/\mu^k$, where $\sigma_J$ is the maximum singular value of $\vb J$). Since each multiplication with $\vb J$ involves $\order{nm}$ flops, if a small number of conjugate gradient and quadratic penalty update steps are taken relative to $m$, then the projection retraction overall requires $\order{nm}$ flops. Algorithm \ref{alg:equality_projection_retraction} shows the skeleton of the projection retraction routine employed in LFPSQP.

\begin{algorithm}
\caption{Projection retraction}
\label{alg:equality_projection_retraction}
\begin{algorithmic}[1]
\Procedure{ProjectRetract}{$\vb x^i$, $\Delta \vb x$, $\vb c$, $\vb J$, $\epsilon_c$, $k_\mathrm{max}$, $\mu^0$}

\State $\vt x \gets \vb x^i + \Delta \vb x$
\State $k \gets 0$
\State $\hat{\vb x}^{k} \gets \vt x$

\State $\vb c^k \gets \vb c(\hat{\vb x}^k)$

\While{$\norm{\vb c^k}_\infty > \epsilon_c$ and $k < k_\mathrm{max}$}
	\State Calculate $\vb J(\hat{\vb x}^k)$
	\State Solve equation \ref{eq:penalty_lm} for $\vb p$ to an absolute residual tolerance of $\epsilon_c$
	
	\State $\hat{\vb x}^{k+1} \gets \textsc{LineSearch}(\phi_{\mu^k}, \alpha \mapsto \hat{\vb x}^{k} + \alpha\vb p)$
	
	\State $k \gets k + 1$
	\State $\vb c^k \gets \vb c(\hat{\vb x}^k)$
	\State $\mu^k \gets \norm{\vb c^k}_2$
\EndWhile
\\
\If{$k = k_\mathrm{max}$}
	\Comment Convergence not established by $k_\mathrm{max}$ iterations
	\State \textbf{return error}
\EndIf
\\

\State $\vb x^{i+1} \gets \hat{\vb x}^k$
\State \textbf{return} $\vb x^{i+1}$

\EndProcedure
\end{algorithmic}
\end{algorithm}

Interestingly, we have found the generation of an analytically invertible preconditioner that approximates $\vb J$ in equation \ref{eq:penalty_lm} with $\vb U$ from the SVD decomposition of $\vb J(\vb x^i)$ does not appear to be useful, even for small step sizes $\norm{\Delta \vb x}_2$. The construction of a simple, general preconditioner for the inner conjugate gradient routine of this projection retraction remains an open problem and would be interesting to consider in future work.

\section{Inequality Constraints}
We now turn our attention to the full, mixed equality and inequality constrained optimization problem in equation \ref{eq:mixed_opt}. Similar to the approach taken by Gabay \cite{gabay1982} and several others, the (smooth) inequality constraints and bound constraints can be transformed to (smooth) equality constraints yielding an equivalent problem in $2n'$ variables with $n' + m'$ equality constraints (where $n' = n + p$ and $m' = m + p$) of the form:
\begin{equation}
\begin{array}{rl}
\displaystyle \min_{\vb x, \vb y \in \R^{n+p}} & f'(\vb x) \\
\mathrm{s.t.} & \vb c'(\vb x) = \bm 0 \\
& \vb h(\vb x, \vb y) = \bm 0.
\end{array}
\label{eq:transformed_mixed_opt}
\end{equation}
Letting $\vb x_{1:n}$ represent a ``slice'' of the first $n$ variables of $\vb x$, the functions in problem \ref{eq:transformed_mixed_opt} are defined as
\begin{align*}
f' &: \vb x \mapsto f(\vb x_{1:n}) \\
\vb c' &: \vb x \mapsto \mqty[\vb c(\vb x_{1:n}) \\ \vb d(\vb x_{1:n}) - \vb x_{n+1:n+p}].
\end{align*}
The variables $\vb x_{n+1:n+p}$ can be considered slack variables. Additionally, bounds on $\vb x$ are concatenated with those that were placed on $\vb d$ prior to problem transformation as $\vb l = (\vb x^l, \vb d^l)$ and $\vb u = (\vb x^u, \vb d^u)$. It is assumed that all entries of $\vb u$ are greater than or equal to those of $\vb l$; otherwise, the problem can immediately be deemed infeasible. The function $\vb h: \R^{n+p} \times \R^{n+p} \to \R^{n+p}$ is responsible for maintaining the box constraints $\vb l \leq \vb x \leq \vb u$ and is defined as
\begin{equation}
h_k : (\vb x, \vb y) \mapsto q_k(x_k - r_k)^2 + (1 - q_k^2)x_k + s_k(y_k - r_k)^2 - (1-s_k^2)y_k - t_k
\end{equation}
for $k \in \{1, \ldots, n+p\}$ and for constant vectors $\vb q$, $\vb r$, $\vb s$, $\vb t \in \R^{n+p}$ defined as
\begin{equation}
\begin{array}{lllllll}
q_k = 0, & r_k = 0, & s_k =\phantom{-} 0, & t_k = 0 & \text{if} & l_k = -\infty, & u_k = \infty \\
q_k=0, & r_k = l_k, & s_k = -1, & t_k = l_k & \text{if} & l_k \in \R, & u_k = \infty \\
q_k=0, & r_k = u_k, & s_k = \phantom{-}1, & t_k = u_k & \text{if} & l_k = -\infty, & u_k \in \R \\
q_k=1, & r_k = \frac{u_k + l_k}{2}, & s_k = \phantom{-}1, & t_k = \frac{(u_k - l_k)^2}{4} & \text{if} & l_k \in \R, & u_k \in \R. \\
\end{array}
\end{equation}
These different conditions on the coefficients in each function $h_k$ represent the different box constraints imposed on $\vb x$. If $x_k$ is unconstrained, then $x_k = y_k$. If $x_k$ exhibits only one finite bound, then $x_k$ is a parabola with respect to $y_k$ with a lower (upper) bound of $l_k$ ($u_k$). Finally, if both lower and upper bounds on $x_k$ are finite, then $h_k$ defines a circle in $(x_k, y_k)$ space between the respective lower and upper bounds. In this way, $\vb h$ avoids the need for two multipliers for every variable's box constraint.

The sparsity of the constraints $\vb h$ can be exploited to yield an algorithm with similar asymptotic complexity compared to the underlying routines of the strictly equality constrained case detailed in section 3 as long as $p = \order{n}$. Namely, with $m' = m + p$ nonlinear constraints $\vb c'$, the outer optimization loop still involves a factorization requiring $\order{n m'^2}$ flops and a retraction subroutine with inner steps requiring $\order{nm'}$ flops.

\subsection{Step Generation and Linesearch}

The Jacobian of $\vb h$ with respect to $(\vb x, \vb y)$, denoted $\vb K$, can be calculated analytically as
\begin{equation}
\vb K^\T(\vb x, \vb y) = \mqty[ \diag(2\vb q \odot(\vb x - \vb r) + (\vb 1 - \vb q \odot \vb q)) \\ \diag(2\vb s \odot(\vb y - \vb r) - (\vb 1 - \vb s \odot \vb s)) ],
\end{equation}
where ``$\odot$'' represents element-wise multiplication and ``$\diag(\vb v)$'' denotes a diagonal matrix with diagonal given by $\vb v$. An orthonormal basis for $\vb K^\T$ can also be readily computed by simply normalizing each of the columns, resulting in the decomposition: 
\begin{equation}
\vb K^\T = \mqty[\vb D_x \\ \vb D_y] \vb S,
\end{equation}
where $\vb S$ is a diagonal matrix containing the normalization constants for each column and the blocks corresponding to the $\vb x$ and $\vb y$ variables are labeled as such. Notably, because $\vb D_x$, $\vb D_y$, and $\vb S$ are all $n' \times n'$ diagonal matrices, they only require $\order{n}$ storage and $\order{n}$ flops for multiplication.

Let $\vb J$ continue to represent the Jacobian of $\vb c$ with respect to the $\vb x$ variables. Instead of computing an SVD of the entire $(n'+m') \times 2n'$ Jacobian of all the constraints, it is useful to construct the following decomposition:
\begin{equation}
\bm{\mathcal J}^\T = \mqty[\grad_{\vb x} \vb h & \vb J^\T \\ \grad_{\vb y} \vb h & \bm 0] = \mqty[\vb D_x & \vb U_x \\ \vb D_y & \vb U_y] \mqty[ \vb S & \vb R \\ \bm 0 & \bm \Sigma \vb V^\T],
\label{eq:inequality_decomp}
\end{equation}
where
\begin{equation}
\mqty[\vb U_x \\ \vb U_y] \bm \Sigma \vb V^\T = \left( \vb I - \mqty[\vb D_x \\ \vb D_y] \mqty[\vb D_x \\ \vb D_y]^\T \right) \mqty[\vb J^\T \\ \bm 0]
\end{equation}
is an SVD decomposition of the \textit{projected} Jacobian requiring $\order{n'm'^2}$ flops and
\begin{equation}
\vb R = \vb D_x \vb J^\T.
\end{equation}
``$\grad_{\vb x}$'' denotes the gradient with respect to the $\vb x$ variables. One can view equation \ref{eq:inequality_decomp} as a sort of ``block QR'' decomposition. Importantly, the matrix containing the $\vb D$ and $\vb U$ blocks serves as an orthonormal basis for the full problem Jacobian. Assuming that $\min_k (u_k - l_k) > \epsilon_\mathrm{rank}$ such that the entries on the diagonal of $\vb S$ are all greater than $\epsilon_\mathrm{rank}$, the numerical rank can still be computed via the singular values of $\bm \Sigma$ as
\begin{equation}
\rank_{\epsilon_\mathrm{rank}}(\bm{\mathcal J}) = n' + \rank_{\epsilon_\mathrm{rank}}(\bm \Sigma) \equiv n' + r.
\end{equation}
Thus, although more storage for the SVD and $\vb R$ is required compared to the factorization featured in section 3, the same asymptotic complexity with respect to $n$ and $m'$ is maintained assuming that $p = \order{n}$.

Let $\vb z^i = (\vb x^i, \vb y^i)$ be the concatenated iterate of $\vb x$ and $\vb y$ at the $i$-th step of the outer optimization loop, and let $\vb U$ now denote the $2n' \times m'$ matrix $\mqty[\vb U_x^\T & \vb U_y^\T]^\T$. Furthermore, let $\mathcal M$, now, denote the feasible set $\mathcal M = \{(\vb x, \vb y) \mid \vb c'(\vb x) = \bm 0, \vb h(\vb x, \vb y)  = \bm 0 \}$. As before, it is assumed that $\mathcal M$ is a smooth submanifold of Euclidean space, with $\bm{\mathcal J}$ full-rank at all points in $\mathcal M$ being sufficient for such a claim. With the above decomposition of $\bm{\mathcal J}$, the projected gradient of the objective function is now given by
\begin{equation}
\mathrm{grad}\, f'(\vb z^i) = P_{\vb z^i}[\grad_{\vb z} f'(\vb x^i)] = \left( \vb I - \mqty[\vb D_x \\ \vb D_y] \mqty[\vb D_x \\ \vb D_y]^\T - \vb U_{:,1:r} \vb U_{:,1:r}^\T \right)\mqty[\grad_{\vb x} f'(\vb x^i) \\ \bm 0].
\end{equation}
The generation of the gradient search direction in Algorithm \ref{alg:gradient} can be easily modified to handle updating $\vb z^i$ instead of $\vb x^i$ alone as shown in Algorithm \ref{alg:inequality_gradient}.

\begin{algorithm}
\caption{Generation of the gradient search direction}
\label{alg:inequality_gradient}
\begin{algorithmic}[1]
\Procedure{GenerateGradientDirection}{$\vb z^i$, $f'$, $\vb J$}

\State Generate decomposition \ref{eq:inequality_decomp} at $\vb z^i$
\State $\Delta \vb z \gets P_{\vb z^i}[-\grad_{\vb z} f'(\vb x^i)]$

\State \textbf{return} $\Delta \vb z$

\EndProcedure
\end{algorithmic}
\end{algorithm}

For convenience, let $\bm \lambda^i = (\bm \lambda^{h,i}, \bm \lambda^{c',i})$ represent the concatenation of multipliers associated with $\vb h$ and $\vb c'$, respectively. Analagous to equation \ref{eq:lambda_equality} and using the decomposition in equation \ref{eq:inequality_decomp},
\begin{equation}
\mqty[\bm \lambda^{h,i} \\ \bm \lambda^{c',i}] = -(\bm{\mathcal J}^\T)^\dagger \grad_{\vb z} f'(\vb x^i) = -\mqty[\vb S^{-1} & -\vb S^{-1} \vb R \vb V_{:,1:r} \bm \Sigma_{1:r,1:r}^{-1} \\ \bm 0 & \vb V_{:,1:r}\bm \Sigma_{1:r,1:r}^{-1} ] \mqty[\vb D_x \grad_{\vb x} f'(\vb x^i) \\ \vb U_{x,:,1:r}^\T \grad_{\vb x} f'(\vb x^i)],
\end{equation}
where $\vb A^\dagger$ represents the pseudo-inverse of $\vb A$. Given $\grad_{\vb x} f'$, $\bm \lambda^i$ is computable in $\order{n'm'}$ flops with the dominant computations being matrix multiplications with the dense matrices $\vb R$ and $\vb V$.

Analogous to equation \ref{eq:equality_saddle}, the Newton search direction saddle point problem becomes:
\begin{equation}
\mqty[\vb W'(\vb x^i, \bm \lambda^{c',i}) + \vb H_x(\bm \lambda^{h,i}) & \bm 0 & \vb D_x & \vb U_{x,:,1:r} \\ 
\bm 0 & \vb H_y(\bm \lambda^{h,i}) & \vb D_y & \vb U_{y,:,1:r} \\
\vb D_x & \vb D_y & \bm 0 & \bm 0 \\
\vb U_{x,:,1:r}^\T & \vb U_{y,:,1:r}^\T & \bm 0 & \bm 0]
\mqty[\Delta \vb x \\ \Delta \vb y \\ \Delta \bm \lambda^{h} \\ \Delta \bm \lambda^{c'}] = \mqty[-P_{\vb z^i}[\grad_{\vb z} f'(\vb x^i)] \\ \vdots \\ \bm 0 \\ \bm 0].
\label{eq:inequality_saddle}
\end{equation}
Here, $\vb W'$ is defined analagously to $\vb W$ in section 3 except with $f'$ and $\vb c'$ and with second derivatives with respect to the $\vb x$ variables. Additionally, $\vb H_x = 2\diag(\bm \lambda^{h,i} \odot \vb q)$ and $\vb H_y = 2\diag(\bm \lambda^{h,i} \odot \vb s)$ are diagonal matrices representing second-order derivatives of $\vb h$. Notably, the only dense subblocks of the matrix on the left-hand side of equation \ref{eq:inequality_saddle} are those involving $\vb W'$ and the $\vb U$ matrices. Thus, matrix multiplication in each iteration of the projected conjugate gradient used to solve equation \ref{eq:inequality_saddle} only requires $\order{n'm'}$ flops. Algorithm \ref{alg:inequality_newton_step} is analogous to Algorithm \ref{alg:newton_step} and details the generation of an inexact Newton search direction based on equation \ref{eq:inequality_saddle}.

\begin{algorithm}
\caption{Generation of the inexact Newton search direction}
\label{alg:inequality_newton_step}
\begin{algorithmic}[1]
\Procedure{GenerateInexactNewtonDirection}{$\vb z^i$, $f'$, $\vb J$, $\delta$}

\State Generate decomposition \ref{eq:inequality_decomp} at $\vb z^i$
\State Solve equation \ref{eq:inequality_saddle} to an absolute residual tolerance of $\delta$

\State \textbf{return} $\Delta \vb z$

\EndProcedure
\end{algorithmic}
\end{algorithm}

\subsection{Retractions}
The retractions described in the strictly equality-constrained case can also be adapted to take advantage of the sparsity of the full Jacobian $\bm{\mathcal J}$ of the augmented problem with respect to $(\vb x, \vb y)$.

\subsubsection{Quasi-Newton Retraction}

Since the functional form of $\vb h$ is known, specific and computationally efficient retractions can be developed for $\mathcal M^h = \{ (\vb x, \vb y) \mid \vb h(\vb x, \vb y) = \bm 0 \}$, the feasible set of $\vb h$, to be used in conjunction with the quasi-Newton scheme discussed in section 3.3.1 for the strictly equality constrained case. Similarly, it will be assumed for this retraction that  $\bm{\mathcal J}$ is (numerically) full-rank (i.e., the bound constraints and the nonlinear constraints are together linearly independent).

With all elements of $\vb u$ strictly greater than $\vb l$, the feasible set associated with $\vb h$,  $\mathcal M^h = \{ (\vb x, \vb y) \mid \vb h(\vb x, \vb y) = \bm 0 \}$, is a submanifold of Euclidean space since the Jacobian is full-rank for all points in $\mathcal M^h$. A retraction
\begin{equation}
R^h_{\vb z^i} : T_{\vb z^i} \mathcal M^h \to \mathcal M^h, (\Delta \vb x, \Delta \vb y) \mapsto (x^{i+1}_1, \ldots, x^{i+1}_{n'}, y^{i+1}_1, \ldots, y^{i+1}_{n'})
\end{equation}
can be constructed in a coordinate-wise manner that takes advantage of the structure of $\vb h$, where each pair $(x_k^{i+1}, y_k^{i+1}$ of the retraction's output is given by
\begin{equation}
\begin{split}
(x_k^{i+1}, y_k^{i+1} ) &= \begin{cases}
(x_k^i + \Delta x_k,  y_k^i + \Delta y_k) & l_k = -\infty, u_k = -\infty \\
(x_k + \Delta x_k + \gamma_k \xi_{k,x}, y_k + \Delta y_k + \gamma_k \xi_{k,y}) & \text{one of } l_k, u_k \text{ is infinite} \\
(r_k, r_k) + \sqrt{t_k} \frac{(x_k^i + \Delta x_k - r_k,  y_k^i + \Delta y_k - r_k)}{\norm{(x_k^i + \Delta x_k - r_k,  y_k^i + \Delta y_k - r_k)}_2} & l_k \in \R, u_k \in \R \\
\end{cases}\\
(\xi_{k,x}, \xi_{k,y}) &\equiv \frac{(-s_k, -2(y_k^i - r_k)}{\norm{(-s_k, -2(y_k^i - r_k)}_2} - (\Delta x_k, \Delta y_k) \\
\gamma_k &\text{ satisfies } h_k(x_k^{i+1}, y_k^{i+1}) = 0
\end{split}
\label{eq:y_retractions}
\end{equation}
for $k \in \{1, \ldots, n' \}$. Again, $\Delta \vb z \in T_{\vb z^i} \mathcal M^h$ is identified with a vector in Euclidean space. The rules above are defined in a piecewise manner to account for different geometries (i.e., line, parabola, and circle) that the box constraints induce. The retraction for a line constraint (first case) is somewhat trivial since a step in the tangent space remains on the line. The retraction for the parabola (second case) involves ``projecting'' the step back onto the parabola along a line to a specially chosen point a unit length ``inward'' to the parabola from $(x_k^i, y_k^i)$ that forms a right angle with the tangent space at $(x_k^i, y_k^i)$. Finally, the retraction for the circle (third case) is a simple bona fide projection back onto the circle. A visualization of these retractions in equation \ref{eq:y_retractions} is shown in Figure \ref{fig:y_retractions}.

\begin{figure}[ht]
\centering
\includegraphics[width=\textwidth]{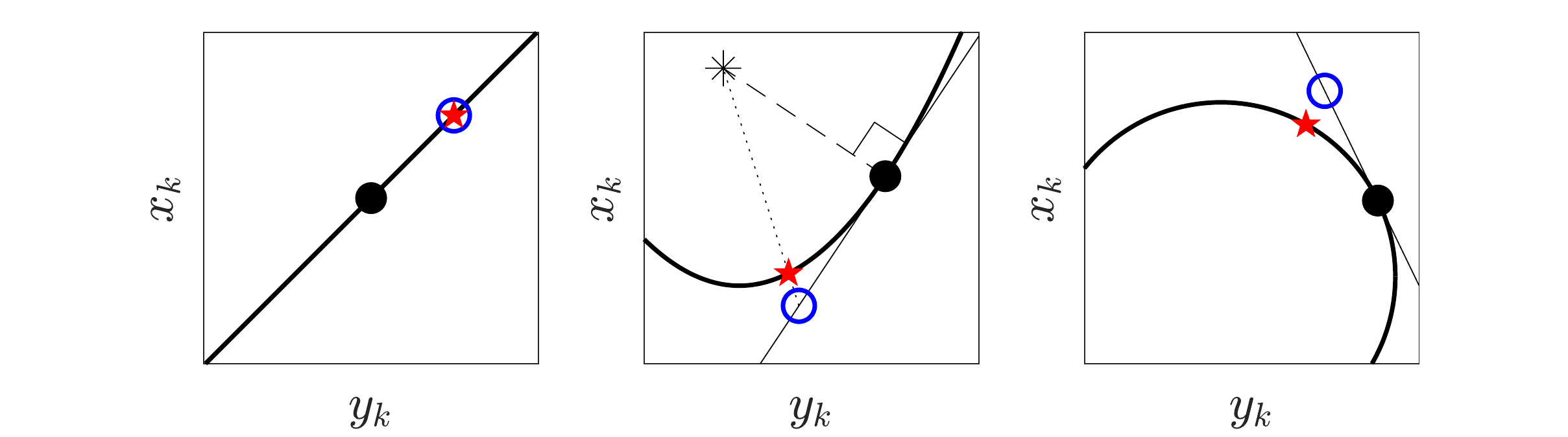}
\caption{Visualization of the coordinate-wise retractions defined in equation \ref{eq:y_retractions}. For each of the three cases, the black point represents $(x_k^i, y_k^i)$, the blue circle represents $(x_k^i + \Delta x_k, y_k^i + \Delta y_k)$, and the red star represents the retraction. The thin solid lines represent the tangent space at $(x_k^i, y_k^i)$. For the parabola (middle), the dotted line toward the asterisk represents the direction $(\xi_{k,x}, \xi_{k,y})$ in equation \ref{eq:y_retractions}.}
\label{fig:y_retractions}
\end{figure}

Given this retraction for $\mathcal M^h$ associated with $\vb h$, we can construct a composite retraction for iterates on $\mathcal M \subset \mathcal M^h$ as
\begin{equation}
\begin{split}
R_{\vb z^i} &: \Delta \vb z \mapsto R^h_{\vb z^i}(\Delta \vb x + \vb U_x \vb w, \Delta \vb y + \vb U_y \vb w), \\
&\qquad \qquad \vb w \text{ s.t. } \vb c'(R^h_{\vb z^i}(\Delta \vb x + \vb U_x \vb w, \Delta \vb y + \vb U_y \vb w)) = \bm 0.
\end{split}
\label{eq:inequality_newton_retraction}
\end{equation}
This retraction is quite similar to that of equation \ref{eq:equality_newton_retraction} except that satisfaction of the $\vb h$ constraints is ``automatically'' handled via composition with $R^h_{\vb z^i}$. Note that $(\Delta \vb x + \vb U_x \vb w, \Delta \vb y + \vb U_y \vb w)$ is indeed in $T_{\vb z^i} \mathcal M^h$ due to the orthnormality of $\grad_{\vb z} \vb h$ and $\vb U$ in the construction of the decomposition in equation \ref{eq:inequality_decomp}. In some sense, one can think of the above retraction as orthographic with respect to the $\vb c$ constraints and ``projective'' with respect to the $\vb h$ constraints. In fact, using the retractor formalism of Absil and Malick  \cite{absil2012}, the above retraction is second-order by the following proposition.

\begin{proposition}
The Newton retraction of equation \ref{eq:inequality_newton_retraction} is a second-order retraction.
\end{proposition}
\begin{proof}
Consider the retractor
\begin{equation}
D : T\mathcal M \to \mathrm{Gr}(n' - m'), (\vb z, \Delta \vb z) \mapsto \spn({\vb U}) + \sum_{k = 1}^{n'} V_k(\vb z, \Delta \vb z)
\end{equation}
where $T\mathcal M$ is the tangent bundle of $\mathcal M$ and $\mathrm{Gr}(n' - m')$ is the set of all $(n' - m')$-planes in $\R^{2n'}$ known as the Grassmann manifold. The use of ``$+$'' and the summation symbol represent direct sums of the relevant linear spaces, and $V_k(\vb z, \Delta \vb z) = \{ a (\ldots, 0, v_{k,x}, 0, \ldots, 0, v_{k,y}, 0, \ldots) \mid a \in \R  \}$, where
\begin{equation}
(v_{k,x}, v_{k,y}) = \begin{cases}
(D_{x,k}, D_{y,k}) & l_k = -\infty, u_k = -\infty \\
(\xi_{k,x}, \xi_{k,y}) & \text{one of } l_k, u_k \text{ is infinite} \\
(x_k^i + \Delta x_k - r_k,  y_k^i + \Delta y_k - r_k) & l_k \in \R, u_k \in \R \\
\end{cases}.
\end{equation}
It can be shown that $D$ is continuous in both $\vb z$ and $\Delta \vb z$. Furthermore,
\begin{equation*}
\sum_{k=1}^{n'} V_k(\vb z, \bm 0_{\vb z}) = \spn\left(\mqty[\vb D_x \\ \vb D_y] \right),
\end{equation*}
implying that $D(\vb z, \bm 0_{\vb z}) = N_{\vb z}\mathcal M$ for all $\vb z \in \mathcal M$. Thus, by theorem 22 of ref. \cite{absil2012}, equation \ref{eq:inequality_newton_retraction} is a second-order retraction induced by $D$.
\end{proof}

Similar to the quasi-Newton procedure detailed in section 3.3.1, the $m'$ coefficients, $\vb w$, in equation \ref{eq:inequality_newton_retraction} can also be found via a quasi-Newton procedure. At each ``inner'' quasi-Newton iteration, $R_{\vb z^i}^h$ can be computed in $\order{n'}$ flops, multiplication of $\vb w$ by $\vb U$ requires $\order{n' m'}$ flops, and Broyden updates require $\order{m'^2}$ flops, resulting overall in $\order{n'm'}$ flops per ``inner'' iteration.

\subsubsection{Projection Retraction}

The projection retraction in the context of problem \ref{eq:transformed_mixed_opt} is completely analogous to that of section 3.3.2 except that the projection objective function becomes
\begin{equation}
\phi_{\mu}(\vb z) = \frac{\mu}{2}\norm{\vb z - \vt z}_2^2 + \frac{1}{2}\norm{\vb c'(\vb x)}_2^2 + \frac{1}{2}\norm{\vb h(\vb x, \vb y)}_2^2,
\end{equation}
where $\vt z = \vb z + \Delta \vb z$ (identifying $\Delta \vb z$ as a vector in Euclidean space), and the Gauss-Newton step becomes
\begin{equation}
\left( \bm{\mathcal J}^\T(\hat{\vb z}^{k}) \bm{\mathcal J}(\hat{\vb z}^{k}) + \mu^k \vb I \right) \vb p = -\left( \bm{\mathcal J}^\T(\hat{\vb z}^{k}) \mqty[\vb c'(\hat{\vb x}^{k}) \\ \vb h(\hat{\vb x}^k, \hat{\vb y}^k)] + \mu^k(\hat{\vb z}^{k} - \vt z) \right)
\end{equation}
at each ``inner'' iteration $k$. Due to the structure of $\bm{\mathcal J}$, multiplication by $\bm{\mathcal J}$ requires $\order{n'm'}$ flops, resulting in a similar asymptotic scaling for each ``inner'' conjugate gradient step as the projection retraction in section 3.3.2.

\section{Numerical Examples}

The following are a few selected examples that showcase the performance of LFPSQP.

\subsection{Rayleigh quotient}

The minimization (maximization) of the Rayleigh quotient on a sphere is a classic Riemannian optimization problem \cite{edelman1998, absil2007, smith1994}. Given a symmetric matrix, $\vb A$, the problem
\begin{equation}
\begin{array}{rl}
\displaystyle \min_{\vb x \in \R^n} & \frac{1}{2} \vb x^\T \vb A \vb x \\
\mathrm{s.t.} & \vb x \vdot \vb x - 1 = 0
\end{array}
\end{equation}
yields the minimum eigenvalue-eigenvector pair. We ran the LFPSQP algorithm with the projection retraction ($\mu^0 = 0.01$), an Armijo line search ($\alpha_0 = 1, s = 0.5$), and search directions produced by the truncated/inexact Newton scheme of Algorithm \ref{alg:newton_step}. The tolerance for the constraint was set to $\epsilon_c = 10^{-6}$. The initial guess, $\vb x^0$, was set to a random vector on the unit sphere, $S^{n-1}$. Let $\vb v$ denote the true eigenvector associated with the minimum eigenvalue. First, for the matrix $\vb A = \diag((n,n-1,\ldots,1))$ where $n=100$, LFPSQP terminated with a function improvement tolerance between steps of $2.3 \times 10^{-10}$ and a projected gradient norm of $3.6 \times 10^{-7}$ after 8 outer iterations. In each outer iteration, no more than 5 total conjugate gradient ``inner'' iterations were required for each projection back onto the constraint manifold (without exploiting the obvious structure of $S^{n-1}$). Figure \ref{fig:rayleigh_diagonal} shows the progress of the algorithm as a function of the number of outer iterations and the number of cumulative matrix-vector multiplies required for function evaluations, gradient evaluations, and Hessian-vector evaluations used throughout the course of the algorithm and within the inexact Newton search direction scheme. As one can see, more matrix-vector multiplies were required near termination, which makes sense considering the inexact Newton scheme allows for less residual error as the iterates become closer to a local optimum. It is also worth mentioning that for the first 3 iterations with the particular random initial guess employed, a negative direction of curvature was detected by the projected conjugate gradient algorithm, but this clearly did not pose a significant issue to the algorithm.

\begin{figure}[h]
\centering
\includegraphics[width=\textwidth]{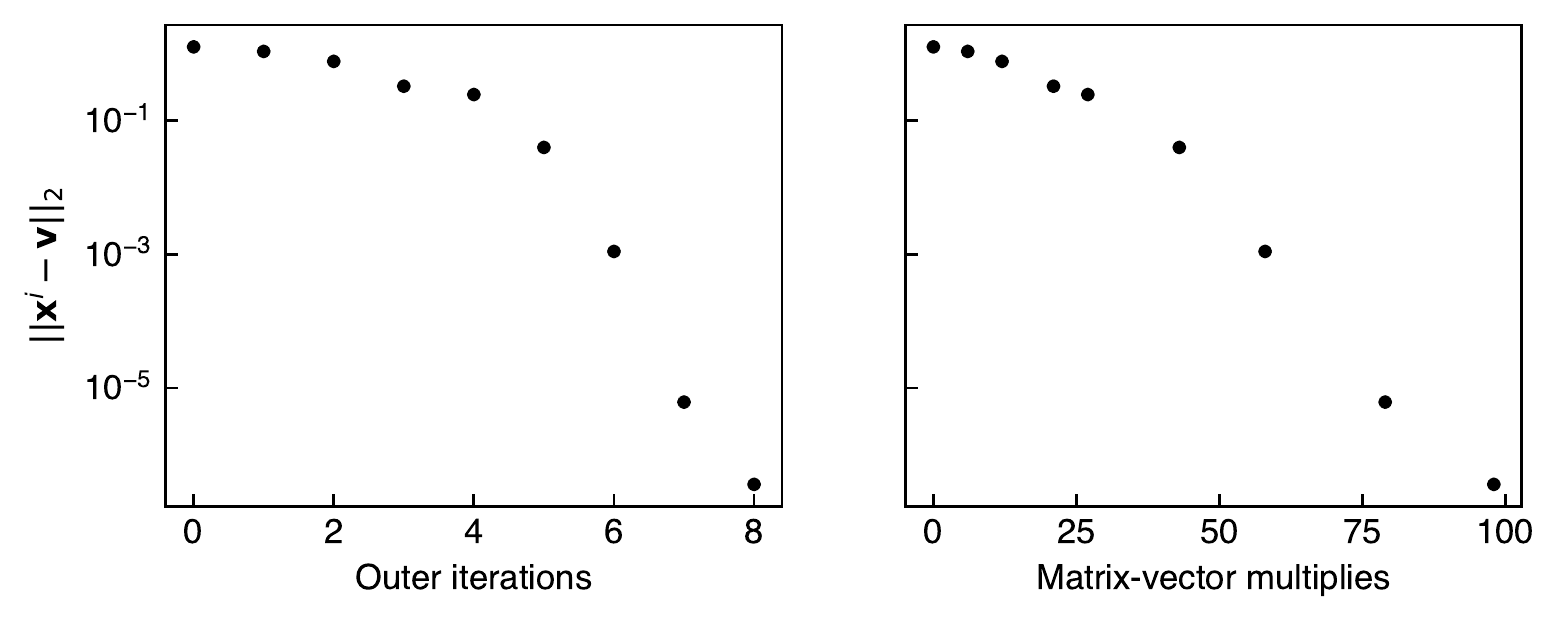}
\caption{The progress of the LFPSQP iterate for the Rayleigh quotient problem with the matrix $\vb A = \diag((n,n-1,\ldots,1))$ where $n=100$ as a function of the number of iterations and as a function of the cumulative number of matrix-vector multiplies with $\vb A$.}
\label{fig:rayleigh_diagonal}
\end{figure}

Figure \ref{fig:rayleigh_random} shows similar results for LFPSQP using a sparse, symmetric matrix $\vb A \in \R^{2000 \times 2000}$ of density approximately 0.02 given by $\vb A = \vb B + \vb B^\T$, where $\vb B$ is a sparse matrix of density 0.01 with nonzero entries drawn from a standard unit normal distribution. LFPSQP terminated after 13 outer iterations with a function improvement tolerance between steps of $6.2 \times 10^{-7}$ and a projected gradient norm of $5.4 \times 10^{-8}$.

\begin{figure}[h]
\centering
\includegraphics[width=\textwidth]{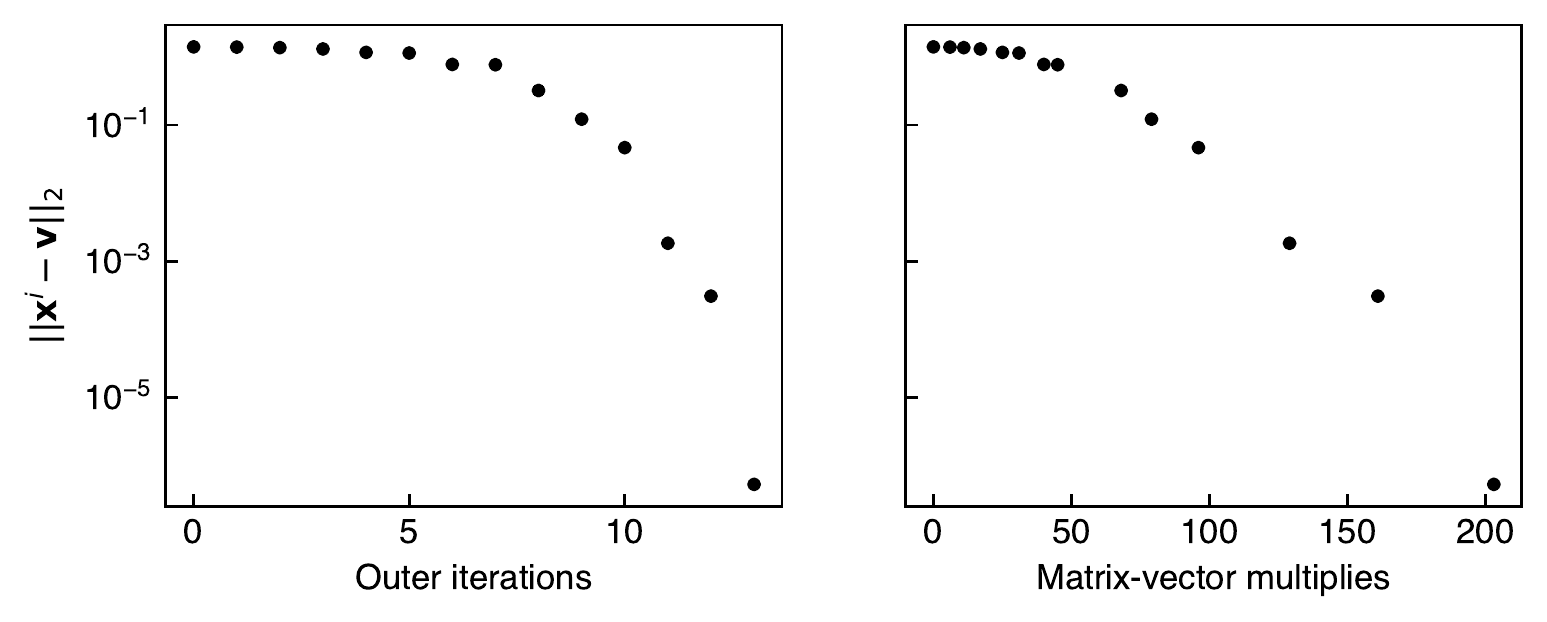}
\caption{The progress of the LFPSQP iterate for the Rayleigh quotient problem with a $2000 \times 2000$ random, symmetric, sparse matrix $\vb A$ as a function of the number of iterations and as a function of the cumulative number of matrix-vector multiplies with $\vb A$.}
\label{fig:rayleigh_random}
\end{figure}

\begin{figure}[h]
\centering
\includegraphics[width=\textwidth]{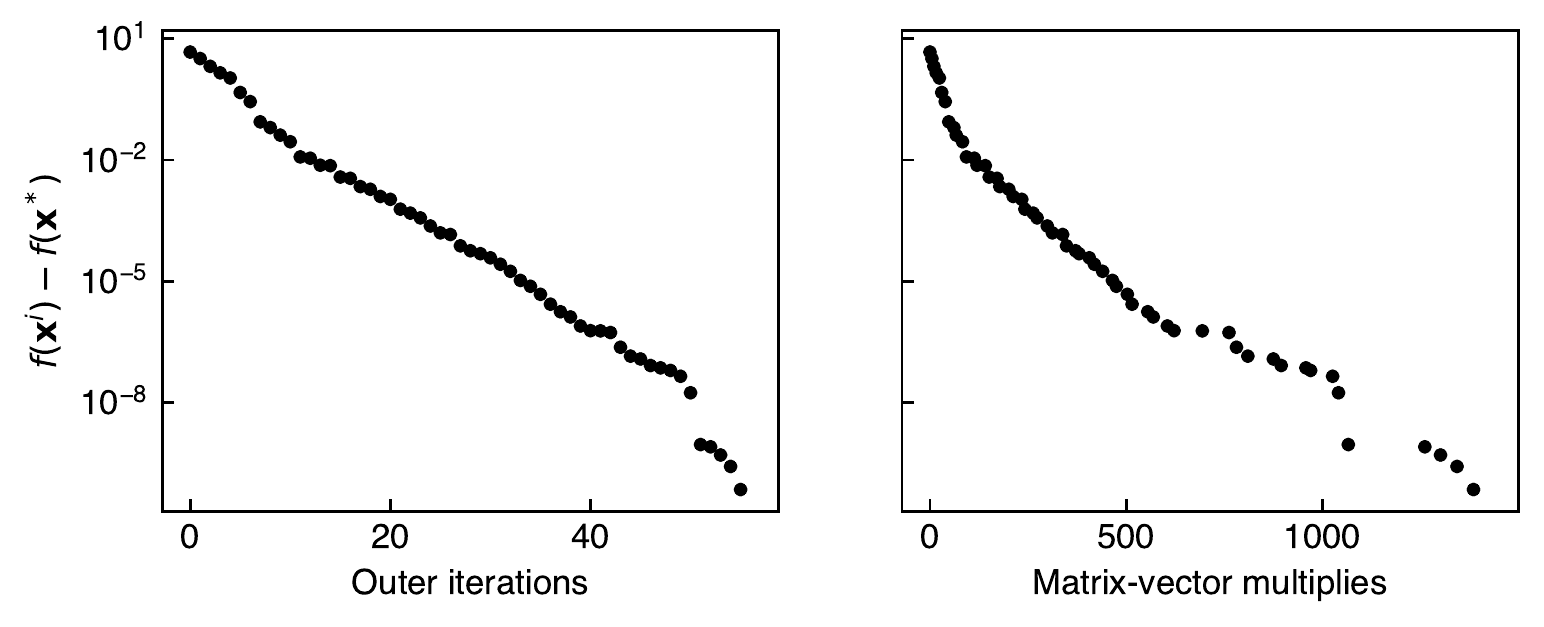}
\caption{The progress of the objective function for the ``positive-orthant'' Rayleigh quotient problem with a $2000 \times 2000$ random, symmetric, sparse matrix $\vb A$ as a function of the number of iterations and as a function of the cumulative number of matrix-vector multiplies with $\vb A$.}
\label{fig:rayleigh_inequality}
\end{figure}

As an example with box constraints, consider the problem of finding the vector that minimizes the Rayleigh quotient with non-negative entries:
\begin{equation}
\begin{array}{rl}
\displaystyle \min_{\vb x \in \R^n} & \frac{1}{2} \vb x^\T \vb A \vb x \\
\mathrm{s.t.} & \vb x \vdot \vb x - 1 = 0 \\
& \vb x \geq \bm 0.
\end{array}
\end{equation}
Geometrically, this problem can be thought of as a minimization of the Rayleigh quotient on the part of the unit sphere that lies in the positive orthant of $\R^n$. Using the same $2000 \times 2000$ random matrix $\vb A$ above, Figure \ref{fig:rayleigh_inequality} shows the progress of the objective function compared to the value at the optimum, which was estimated using the value at termination. The initial guess was set to a positive random vector on the unit sphere, and the tolerance for the constraints was set to a very small value of $\epsilon_c = 10^{-8}$. For certain steps exhibiting significant backtracking, a few hundred cumulative iterations of the ``inner'' projection retraction conjugate gradient routine was required among all calls from the backtracking line search routine. However, given the simplicity of the unit-norm constraint, this did not pose a significant computational burden, as the LFPSQP algorithm took less than 0.5 seconds to execute in Julia. LFPSQP terminated after 56 outer iterations with a function improvement tolerance between steps of $7.1 \times 10^{-11}$ and a projected gradient norm of $1.3 \times 10^{-6}$. The resulting local minimum that was found featured 1008 entries that were greater than $10^{-8}$ in magnitude, meaning about half of the box constraints were active. It is also worth noting that fewer than 2000 matrix-vector multiplies were required to solve this ``positive-orthant'' variant of the Rayleigh-quotient problem. Compared to an interior point method which may solve a system of equations involving $\vb A$ exactly at each step, this number of cumulative matrix-vector multiplies potentially represents a significant saving of computational expense.

\subsection{Simple inequality constraint}
Consider a linear objective function with a feasible set equal to the solid unit sphere:
\begin{equation}
\begin{array}{rl}
\displaystyle \min_{\vb x \in \R^n} & \vb c \vdot \vb x \\
\mathrm{s.t.} & \vb x \vdot \vb x \leq 1 \\
\end{array}
\label{eq:circle_inequality}
\end{equation}
The analytical solution is given by $\vb x^* = -\vb c / \norm{\vb c}_2$. For $n=1000$, a random coefficient vector, $\vb c$ was generated, and LFPSQP was run with an initial guess of $\vb x^0 = \bm 0$ using the quasi-Newton retraction detailed in section 4.2.1. Figure \ref{fig:circle_inequality} shows the Euclidean distance of each LFPSQP iterate from the analytical solution. At each iterate, a maximum of 4 ``inner'' quasi-Newton retraction iterations were required to retract back onto the (augmented) constraint manifold with a tolerance of $\epsilon_c = 10^{-6}$. Overall, after 7 iterations, LFPSQP terminated with a projected gradient norm of $8.6 \times 10^{-9}$. 

\begin{figure}[h]
\centering
\includegraphics{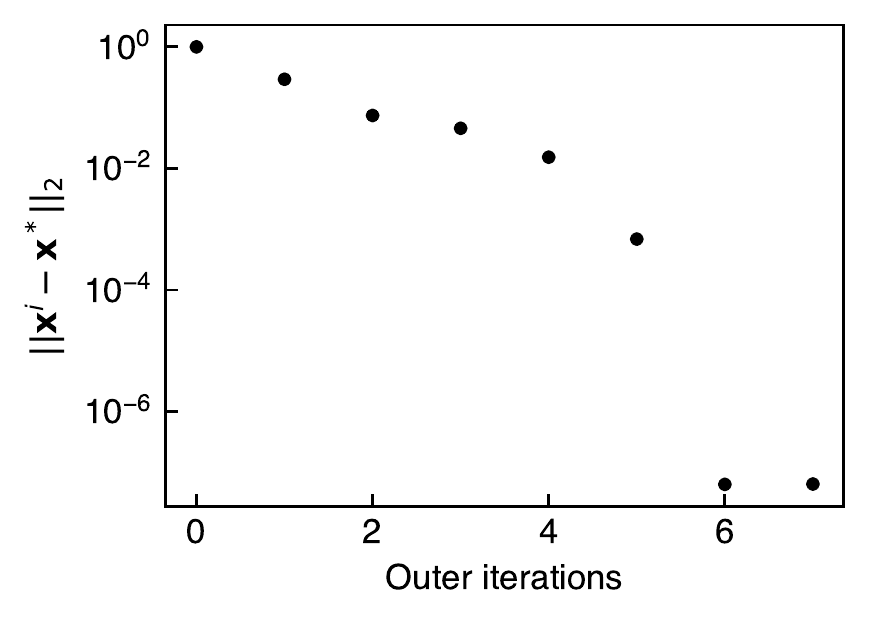}
\caption{The progress of the LFPSQP iterate for problem \ref{eq:circle_inequality} with $n=1000$ compared to the analytical solution, $\vb x^*$.}
\label{fig:circle_inequality}
\end{figure}

\subsection{Degenerate constraints}

Consider the following somewhat pathological, \textit{non-manifold} example:
\begin{equation}
\begin{array}{rl}
\displaystyle \min_{\vb x \in \R^2} & -x_1 - x_2/2\\
\mathrm{s.t.} & x_2 \leq -(x_1 +1)(x_1 - 1)x_1^2 \\
& x_2 \geq (x_1 +1)(x_1 - 1)x_1^2. \\
\end{array}
\label{eq:degenerate_quartic}
\end{equation}
The inequality constraints, here, represent a figure-eight-shaped region that ``pinches'' in the center. Figure \ref{fig:degenerate_quartic} shows the path taken by LFPSQP using the projection retraction with both inexact Newton proposal directions as well as gradient proposal directions, with $\alpha_0$ set to 0.5 for the latter for the purpose of illustrating the projected gradient flow. Clearly, with projected gradient steps, the algorithm traverses the degenerate ``pinch'' in the center, the algorithm for the projection retraction does not falter, and feasibility is maintained throughout. Interestingly, the first inexact Newton step avoids the ``pinch'' altogether with a proposal already lying in the right half of the feasible region. The fact that the iterates ``stick'' to the top boundary of the feasible region in Figure \ref{fig:degenerate_quartic} and do not approach the optimum in a straight line may seem counterintuitive but is due to the underlying nonlinearity of the higher-dimensional \textit{augmented} feasible set that is inherently not visualized in Figure \ref{fig:degenerate_quartic}.

\begin{figure}[h]
\centering
\includegraphics{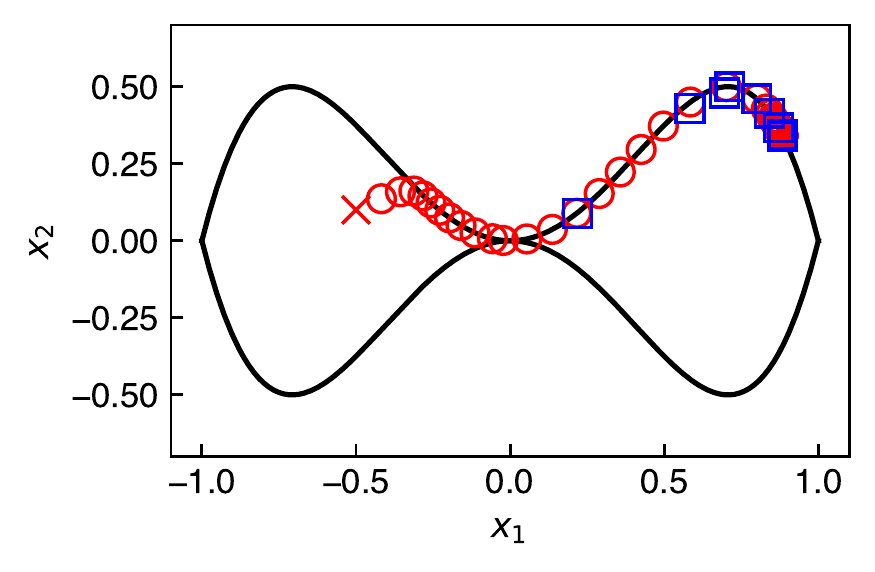}
\caption{LFPSQP iterates using projected gradient propsoals with $\alpha_0 = 0.5$ (red circles) and inexact Newton proposals (blue squares) at each step ($\vb x^0$ is shown with a red ``x'') in the course of solving problem \ref{eq:degenerate_quartic}. The feasible region is outlined in black.}
\label{fig:degenerate_quartic}
\end{figure}

A similarly degenerate example is the following problem:
\begin{equation}
\begin{array}{rl}
\displaystyle \min_{\vb x \in \R^2} & -x_1 - x_2/2\\
\mathrm{s.t.} & \cos^2(x_1) + x_2^2 \leq 1 \\
& -2 \leq x_1 \leq 2
\end{array}
\label{eq:degenerate_cos}
\end{equation}
with analogous LFPSQP results presented in Figure \ref{fig:degenerate_cos}. Again, the LFPSQP iterates for both gradient proposal and inexact Newton proposal directions are able to traverse the degenerate ``pinch'' point to the optimum, despite the different nature of the degeneracy in problem \ref{eq:degenerate_cos} compared to problem \ref{eq:degenerate_quartic} (i.e., the Jacobian of the active constraints at the origin is rank zero for problem  \ref{eq:degenerate_cos} as opposed to rank one for problem \ref{eq:degenerate_quartic}).

\begin{figure}[h]
\centering
\includegraphics{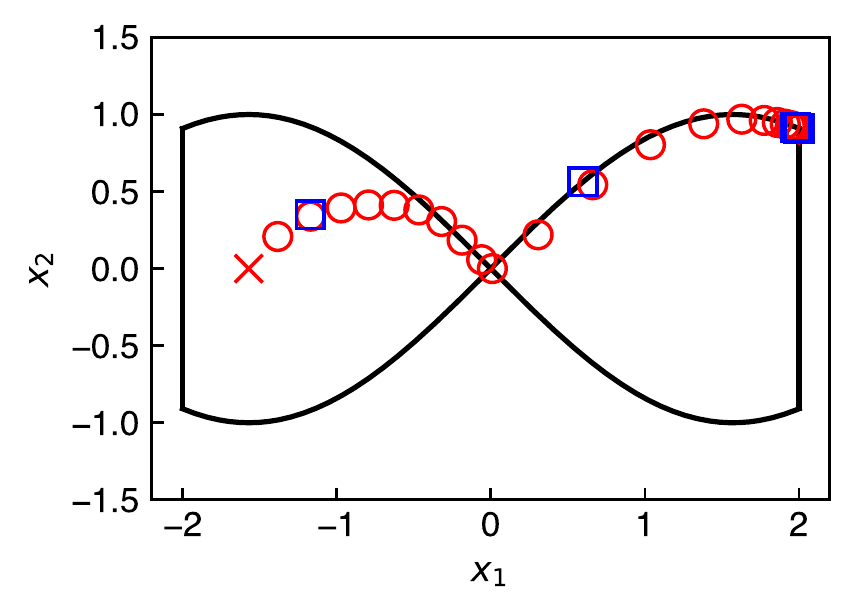}
\caption{LFPSQP iterates using projected gradient propsoals with $\alpha_0 = 0.5$ (red circles) and inexact Newton proposals (blue squares) at each step ($\vb x^0$ is shown with a red ``x'') in the course of solving problem \ref{eq:degenerate_cos}. The feasible region is outlined in black.}
\label{fig:degenerate_cos}
\end{figure}

When applied to both of these problems, interior point methods may get trapped at the degenerate ``pinch`` point before reaching the optimum at the right side of the feasible regions. Indeed, experiments using MATLAB's \texttt{fmincon} interior point solver with default settings and termination criteria indicate this is the case for certain (but not all) initial points. Although challenging to demonstrate theoretically, the above results seem to indicate empirically that LFPSQP may be robust enough to handle certain problems featuring non-manifold/degenerate constraints or other pathological features that may violate the assumptions imposed on the objective and constraint functions throughout the work.

\section{Discussion}

Future work that takes advantage of sparsity in the constraint Jacobian, which is often the case in practice, would obviate the need for a factorization. While factorization via the SVD likely enhances the numerical stability of the method, it is obviously slow compared to iterative methods that excel when faced with problems involving sparse matrices. Such an iterative method (e.g., conjugate gradient of the normal equations) to handle projection steps in all of the subroutines (as opposed to projection with $\vb U$) could potentially prove useful for certain structured and sparse problems, potentially enabling very large-scale problems to be solved efficiently.

It was also implicitly assumed that initial iterates, $\vb x^0$ are feasible. Although not discussed, without a feasible initial iterate, one could envision performing a ``phase 1'' optimization in order to find a feasible point or demonstrate that the problem is infeasible, as is standard in the linear programming literature \cite{bertsimas1997}.

\section{Conclusions}
In this work, a locally feasibly projected sequential quadratic programming (LFPSQP) algorithm was developed to perform feasible optimization with nonlinear objective functions on arbitrary, implicitly defined nonlinear constraint manifolds. Each outer iteration step of the algorithm is largely dominated by an $\order{nm^2}$ singular value decomposition (where $n$ is the number of independent variables and $m$ is the number of nonlinear constraints imposed) that enables numerically stable projection operations in various subroutines and allows robust constraint degeneracy detection. Drawing on the Riemannian optimization literature, computationally efficient retractions were constructed with inner steps requiring $\order{nm}$ flops. To handle inequality constraints, an augmented problem with a higher-dimensional constraint manifold was constructed, and retractions that take advantage of the structure of this augmented manifold were constructed. Importantly, the imposition of box constraints in addition to nonlinear constraints does not contribute to the asymptotic complexity of the ``inner'' or ``outer'' steps via the use of a particular decomposition of the full problem Jacobian.

A package, LFPSQP.jl \cite{silmorelfpsqp}, was created in the Julia language and takes full advantage of Julia's automatic differentiation tools. In particular, mixed-mode automatic differentiation was used to avoid the explicit construction of the Hessian of the Lagrangian and is called when calculating a second-order search direction with an inexact Newton, projected conjugate gradient \cite{gould2001, nocedal2006} subroutine.

Future work examining the practicality of LFPSQP on large-scale problems as well as the incorporation of iterative routines to perform projection steps without factorization would likely be fruitful. Furthermore, improvements to the efficiency of the projection retraction inner problem could be interesting to explore.

\section*{Acknowledgments}
The authors would like to thank Paul Barton for helpful discussions.

\bibliographystyle{siam}
\bibliography{extracted.bib}

\begin{thebibliography}{10}

\bibitem{absil2007}
{\sc P.-A. Absil, C.~Baker, and K.~Gallivan}, {\em Trust-{{Region Methods}} on
  {{Riemannian Manifolds}}}, Found Comput Math, 7 (2007), pp.~303--330.

\bibitem{absil2009a}
{\sc P.-A. Absil, R.~Mahony, and R.~Sepulchre}, {\em Optimization
  {{Algorithms}} on {{Matrix Manifolds}}}, {Princeton University Press}, Apr.
  2009.

\bibitem{absil2012}
{\sc P.-A. Absil and J.~Malick}, {\em Projection-like {{Retractions}} on
  {{Matrix Manifolds}}}, SIAM J. Optim., 22 (2012), pp.~135--158.

\bibitem{absil2009}
{\sc P.-A. Absil, J.~Trumpf, R.~Mahony, and B.~Andrews}, {\em All roads lead to
  {{Newton}}: {{Feasible}} second-order methods for equality-constrained
  optimization}, Technical {{Report}} UCL-INMA-2009.024, Aug. 2009.

\bibitem{adler2002}
{\sc R.~L. Adler, J.-P. Dedieu, J.~Margulies, M.~Martens, and M.~Shub}, {\em
  Newton's method on {{Riemannian}} manifolds and a geometric model for the
  human spine}, IMA Journal of Numerical Analysis, 22 (2002), pp.~359--390.

\bibitem{aihara2017}
{\sc K.~Aihara and H.~Sato}, {\em A matrix-free implementation of {{Riemannian
  Newton}}'s method on the {{Stiefel}} manifold}, Optim Lett, 11 (2017),
  pp.~1729--1741.

\bibitem{anitescu2000}
{\sc M.~Anitescu}, {\em Degenerate {{Nonlinear Programming}} with a {{Quadratic
  Growth Condition}}}, SIAM J. Optim., 10 (2000), pp.~1116--1135.

\bibitem{Manoptjl}
{\sc R.~Bergmann}, {\em Manopt.jl}, 2021.
\newblock Available at \url{https://github.com/JuliaManifolds/Manopt.jl}.

\bibitem{bertsekas1976}
{\sc D.~P. Bertsekas}, {\em On {{Penalty}} and {{Multiplier Methods}} for
  {{Constrained Minimization}}}, SIAM J. Control Optim., 14 (1976),
  pp.~216--235.

\bibitem{bertsekas2016}
\leavevmode\vrule height 2pt depth -1.6pt width 23pt, {\em Nonlinear
  {{Programming}}}, {Athena Scientific}, 2016.

\bibitem{bertsimas1997}
{\sc D.~Bertsimas and J.~N. Tsitsiklis}, {\em Introduction to {{Linear
  Optimization}}}, {Athena Scientific}, 1997.

\bibitem{manopt}
{\sc N.~Boumal, B.~Mishra, P.-A. Absil, and R.~Sepulchre}, {\em Manopt, a
  {{Matlab}} toolbox for optimization on manifolds}, Journal of Machine
  Learning Research, 15 (2014), pp.~1455--1459.

\bibitem{broyden1973}
{\sc C.~G. BROYDEN, J.~E. DENNIS, Jr., and J.~J. MOR{\'E}}, {\em On the
  {{Local}} and {{Superlinear Convergence}} of {{Quasi}}-{{Newton Methods}}},
  IMA Journal of Applied Mathematics, 12 (1973), pp.~223--245.

\bibitem{byrd2008}
{\sc R.~H. Byrd, F.~E. Curtis, and J.~Nocedal}, {\em An {{Inexact SQP Method}}
  for {{Equality Constrained Optimization}}}, SIAM J. Optim., 19 (2008),
  pp.~351--369.

\bibitem{cao2017}
{\sc G.~Cao, E.~M.-K. Lai, and F.~Alam}, {\em Gaussian process model predictive
  control of unknown non-linear systems}, IET Control Theory \&amp;
  Applications, 11 (2017), pp.~703--713.

\bibitem{dembo1982}
{\sc R.~S. Dembo, S.~C. Eisenstat, and T.~Steihaug}, {\em Inexact {{Newton
  Methods}}}, SIAM J. Numer. Anal., 19 (1982), pp.~400--408.

\bibitem{dontchev2019}
{\sc A.~L. Dontchev, M.~Huang, I.~V. Kolmanovsky, and M.~M. Nicotra}, {\em
  Inexact {{Newton}}\textendash{{Kantorovich Methods}} for {{Constrained
  Nonlinear Model Predictive Control}}}, IEEE Transactions on Automatic
  Control, 64 (2019), pp.~3602--3615.

\bibitem{dowling2015}
{\sc A.~W. Dowling and L.~T. Biegler}, {\em Degeneracy {{Hunter}}: {{An
  Algorithm}} for {{Determining Irreducible Sets}} of {{Degenerate
  Constraints}} in {{Mathematical Programs}}}, in Computer {{Aided Chemical
  Engineering}}, K.~V. Gernaey, J.~K. Huusom, and R.~Gani, eds., vol.~37 of
  12th {{International Symposium}} on {{Process Systems Engineering}} and 25th
  {{European Symposium}} on {{Computer Aided Process Engineering}}, {Elsevier},
  Jan. 2015, pp.~809--814.

\bibitem{edelman1998}
{\sc A.~Edelman, T.~A. Arias, and S.~T. Smith}, {\em The {{Geometry}} of
  {{Algorithms}} with {{Orthogonality Constraints}}}, SIAM J. Matrix Anal.
  Appl., 20 (1998), pp.~303--353.

\bibitem{gabay1982}
{\sc D.~Gabay}, {\em Minimizing a differentiable function over a differential
  manifold}, J Optim Theory Appl, 37 (1982), pp.~177--219.

\bibitem{golub2013}
{\sc G.~H. Golub and C.~F. Van~Loan}, {\em Matrix Computations}, Johns
  {{Hopkins}} Studies in the Mathematical Sciences, {The Johns Hopkins
  University Press}, {Baltimore}, fourth~ed., 2013.

\bibitem{gould1989}
{\sc N.~I.~M. Gould}, {\em On the {{Convergence}} of a {{Sequential Penalty
  Function Method}} for {{Constrained Minimization}}}, SIAM J. Numer. Anal., 26
  (1989), pp.~107--128.

\bibitem{gould2001}
{\sc N.~I.~M. Gould, M.~E. Hribar, and J.~Nocedal}, {\em On the {{Solution}} of
  {{Equality Constrained Quadratic Programming Problems Arising}} in
  {{Optimization}}}, SIAM J. Sci. Comput., 23 (2001), pp.~1376--1395.

\bibitem{griewank2008}
{\sc A.~Griewank and A.~Walther}, {\em Evaluating Derivatives: Principles and
  Techniques of Algorithmic Differentiation}, {SIAM}, {Philadelphia, PA},
  second~ed., 2008.

\bibitem{huper2004}
{\sc K.~Huper and J.~Trumpf}, {\em Newton-like methods for numerical
  optimization on manifolds}, in Conference {{Record}} of the
  {{Thirty}}-{{Eighth Asilomar Conference}} on {{Signals}}, {{Systems}} and
  {{Computers}}, 2004., vol.~1, Nov. 2004, pp.~136--139 Vol.1.

\bibitem{lawrence2001}
{\sc C.~T. Lawrence and A.~L. Tits}, {\em A {{Computationally Efficient
  Feasible Sequential Quadratic Programming Algorithm}}}, SIAM J. Optim., 11
  (2001), pp.~1092--1118.

\bibitem{luenberger1972}
{\sc D.~G. Luenberger}, {\em The {{Gradient Projection Method Along
  Geodesics}}}, Management Science, 18 (1972), pp.~620--631.

\bibitem{ma2021}
{\sc R.-R. Ma and Z.-J. Bai}, {\em A {{Riemannian}} inexact {{Newton}}-{{CG}}
  method for stochastic inverse singular value problems}, Numer Linear Algebra
  Appl, 28 (2021).

\bibitem{mccormick1983}
{\sc G.~P. McCormick}, {\em Nonlinear Programming: Theory, Algorithms, and
  Applications}, {Wiley}, {New York}, 1983.

\bibitem{nash2000}
{\sc S.~G. Nash}, {\em A survey of truncated-{{Newton}} methods}, Journal of
  Computational and Applied Mathematics, 124 (2000), pp.~45--59.

\bibitem{nocedal2006}
{\sc J.~Nocedal and S.~J. Wright}, {\em Numerical Optimization}, Springer
  Series in Operation Research and Financial Engineering, {Springer}, {New
  York, NY}, second~ed., 2006.

\bibitem{oustry1999}
{\sc F.~Oustry}, {\em The \$\textbackslash{{U}}\$-{{Lagrangian}} of the
  {{Maximum Eigenvalue Function}}}, SIAM J. Optim., 9 (1999), pp.~526--549.

\bibitem{revels}
{\sc J.~Revels}, {\em {{ReverseDiff}}.jl}.
\newblock Available at \url{https://github.com/JuliaDiff/ReverseDiff.jl}.

\bibitem{revels2016}
{\sc J.~Revels, M.~Lubin, and T.~Papamarkou}, {\em Forward-{{Mode Automatic
  Differentiation}} in {{Julia}}}, arXiv:1607.07892 [cs],  (2016).

\bibitem{silmorelfpsqp}
{\sc K.~S. Silmore}, {\em {LFPSQP}.jl}.
\newblock Available at \url{https://github.com/ksil/LFPSQP.jl}.

\bibitem{smith1994}
{\sc S.~T. Smith}, {\em Optimization {{Techniques}} on {{Riemannian
  Manifolds}}}, in Hamiltonian and Gradient Flows, Algorithms, and Control,
  A.~Bloch, ed., no.~v. 3 in Fields {{Institute}} Communications, {American
  Mathematical Society}, {Providence, RI}, 1994.

\bibitem{sun2016}
{\sc Z.~Sun, Y.~Tian, H.~Li, and J.~Wang}, {\em A superlinear convergence
  feasible sequential quadratic programming algorithm for bipedal dynamic
  walking robot via discrete mechanics and optimal control}, Optimal Control
  Applications and Methods, 37 (2016), pp.~1139--1161.

\bibitem{tenny2004}
{\sc M.~J. Tenny, S.~J. Wright, and J.~B. Rawlings}, {\em Nonlinear {{Model
  Predictive Control}} via {{Feasibility}}-{{Perturbed Sequential Quadratic
  Programming}}}, Computational Optimization and Applications, 28 (2004),
  pp.~87--121.

\bibitem{wachter2006}
{\sc A.~W{\"a}chter and L.~T. Biegler}, {\em On the implementation of an
  interior-point filter line-search algorithm for large-scale nonlinear
  programming}, Math. Program., 106 (2006), pp.~25--57.

\bibitem{wright2004}
{\sc S.~J. Wright and M.~J. Tenny}, {\em A {{Feasible Trust}}-{{Region
  Sequential Quadratic Programming Algorithm}}}, SIAM J. Optim., 14 (2004),
  pp.~1074--1105.

\bibitem{yamashita2001}
{\sc N.~Yamashita and M.~Fukushima}, {\em On the {{Rate}} of {{Convergence}} of
  the {{Levenberg}}-{{Marquardt Method}}}, in Topics in {{Numerical Analysis}},
  G.~Alefeld and X.~Chen, eds., vol.~15, {Springer Vienna}, {Vienna}, 2001,
  pp.~239--249.

\bibitem{zhang2010}
{\sc L.-H. Zhang}, {\em Riemannian {{Newton Method}} for the {{Multivariate
  Eigenvalue Problem}}}, SIAM J. Matrix Anal. Appl., 31 (2010), pp.~2972--2996.

\bibitem{zhao2018a}
{\sc Z.~Zhao, Z.-J. Bai, and X.-Q. Jin}, {\em A {{Riemannian}} inexact
  {{Newton}}-{{CG}} method for constructing a nonnegative matrix with
  prescribed realizable spectrum}, Numer. Math., 140 (2018), pp.~827--855.

\end{thebibliography}

\end{document}